\documentclass{amsart}
\usepackage{url,cite,hyperref,amsmath,amsthm,amssymb,mathtools, soul, tikz-cd, float}
\usepackage[margin=1.25in]{geometry}
\usepackage{xcolor}
\usepackage{comment}

\newtheorem{theorem}{Theorem}[section]

\newtheorem{lemma}[theorem]{Lemma}
\theoremstyle{definition}
\newtheorem{definition}[theorem]{Definition}
\newtheorem{remark}[theorem]{Remark}
\newtheorem{example}[theorem]{Example}
\newtheorem{examples}[theorem]{Examples}

\newtheorem{notation}[theorem]{Notation}
\numberwithin{equation}{section}

\newcommand{\E}{\mathcal{E}}
\newcommand{\N}{\mathbb{N}}
\definecolor{cof}{RGB}{219,144,71}
\definecolor{pur}{RGB}{186,146,162}
\definecolor{greeo}{RGB}{91,173,69}
\definecolor{greet}{RGB}{52,111,72}
\pgfdeclarelayer{background}
\pgfsetlayers{background,main}

\begin{document}  

\title{Moves on $k$-graphs preserving Morita equivalence}
\author[C. Eckhardt, K. Fieldhouse, D. Gent, E. Gillaspy, I. Gonzales, and D. Pask]{Caleb Eckhardt, Kit Fieldhouse, Daniel Gent, Elizabeth Gillaspy, \\Ian Gonzales, and David Pask}

\address[C. Eckhardt]{Department of Mathematics,
  Miami University,
  Oxford, OH, 45056,
  United States}
\email{eckharc@miamioh.edu}

\address[K. Fieldhouse]{Department of Mathematical Sciences, University of Montana, 32 Campus Drive \#0864, Missoula, MT 59812, USA}
\email{kit.fieldhouse@umconnect.umt.edu}

\address[D. Gent]{Department of Mathematical Sciences, University of Montana, 32 Campus Drive \#0864, Missoula, MT 59812, USA}
\email{danielpgent@gmail.com}

\address[E. Gillaspy]{Department of Mathematical Sciences, University of Montana, 32 Campus Drive \#0864, Missoula, MT 59812, USA}
\email{elizabeth.gillaspy@mso.umt.edu}

\address[I. Gonzales]{Department of Mathematical Sciences, University of Montana, 32 Campus Drive \#0864, Missoula, MT 59812, USA}
\email{ian.gonzales@umconnect.umt.edu}

\address[D. Pask]{School of Mathematics and Applied Statistics, University of Wollongong, Northfields Ave Wollongong NSW 2522, Australia}
\email{david\underline{\ }pask@uow.edu.au}

\begin{abstract}
We initiate the program of extending to higher-rank graphs ($k$-graphs) the geometric classification of directed graph $C^*$-algebras, as completed in the 2016 paper of Eilers, Restorff, Ruiz, and S\o{}rensen [ERRS16].  To be precise, we identify four ``moves,'' or modifications, one can perform on a $k$-graph $\Lambda$, which leave invariant the Morita equivalence class of its $C^*$-algebra $C^*(\Lambda)$.  These moves -- insplitting, delay, sink deletion, and reduction -- are inspired by the moves for directed graphs described by S\o{}rensen [S\o13] and Bates-Pask [BP04].  Because of this, our perspective on $k$-graphs focuses on the underlying directed graph.  We consequently include two new results, Theorem 2.3 and Lemma 2.9, about the relationship between a $k$-graph and its underlying directed graph.
\end{abstract}

\maketitle

\section{Introduction} \label{sec:intro}

Recent years have seen a number of breakthroughs in the classification of $C^*$-algebras by $K$-theoretic invariants.  For separable simple unital $C^*$-algebras $A$ which have finite nuclear dimension and satisfy the Universal Coefficient Theorem of \cite{Rosenberg87}, the {\em Elliott invariant} (consisting of the ordered $K$-theory of $A$, its trace simplex, and the pairing between traces and $K_0(A)$) is a classifying invariant \cite{tikuisis-white-winter, EllGongLinNiu, GongLinNiu}: two such $C^*$-algebras $A, B$ are isomorphic if and only if their Elliott invariants are isomorphic.  Work has already begun \cite{EllGongLinNiu-KK-contractible, GongLin-stably-projectionless} on expanding these results to the non-unital setting.

The Cuntz--Krieger algebras $\mathcal O_A$ \cite{CK} associated to irreducible matrices $A$ were one of the early classes of $C^*$-algebras for which $K$-theory was shown to be a classifying invariant \cite{CK, franks, rordam-CK}. When $A$ is not irreducible, $\mathcal O_A$ is not simple, leaving these $C^*$-algebras outside the scope of the Elliott classification program.  However, the proof of the $K$-theoretic classification of simple Cuntz--Krieger algebras draws heavily on the dynamical characterization of Cuntz--Krieger algebras as arising from one-sided shifts of finite type \cite{CK}.  As this dynamical characterization holds in the non-simple case as well, Cuntz--Krieger algebras were a natural setting for a first foray into classification of non-simple $C^*$-algebras.  This program was brought to fruition by Eilers, Restorff, Ruiz, and S\o{}rensen in \cite{errs}.

Interpreting $A$ as the adjacency matrix of a directed graph $E_A$, we have a canonical isomorphism $\mathcal O_A \cong C^*(E_A)$.  Using this perspective, as well as techniques from symbolic dynamics \cite{huang, boyle, boyle-huang}, Eilers et al.~obtained both a $K$-theoretic and a graph-theoretic classification of unital graph $C^*$-algebras.  To be precise, \cite{errs} identifies 6 ``moves'' on directed graphs $E$ with finitely many vertices\footnote{A graph $E$ has finitely many vertices iff $C^*(E)$ is unital.} which preserve the stable isomorphism class of $C^*(E).$ The authors then use filtered $K$-theory to show that, for such graphs $E, F$, an isomorphism $C^*(E) \otimes \mathcal{K} \cong C^*(F) \otimes \mathcal{K}$ can only exist if we can pass from $E$ to $F$ by a finite sequence of these 6 moves and their inverses. Eilers et al.~also show in \cite{errs} that isomorphism of two unital graph $C^*$-algebras $C^*(E), C^*(F)$ is equivalent to the existence of an order-preserving isomorphism of the filtered $K$-theory of $C^*(E)$ and $C^*(F)$.

The $K$-theory of  a graph $C^*$-algebra  \cite{cuntz-K-thy, bhrs} dictates that if $C^*(E)$ is simple, it is either approximately finite-dimensional or purely infinite.
Kumjian and Pask developed the theory of higher-rank graphs, or $k$-graphs, in \cite{kp} to provide a broader range of combinatorial examples of $C^*$-algebras.  Formally, a $k$-graph $\Lambda$ is a countable category with a functor $d: \Lambda \to {\mathbb N}^k$ satisfying a factorization property (see Definition \ref{def:kgraphRB} below).  However, $k$-graphs are also closely linked to buildings \cite{robertson-steger, konter-vdovina} and to higher-rank shifts of finite type via textile systems \cite{johnson-madden}.
The graph-theoretic inspiration for higher-rank graphs was made precise by Hazlewood et al.~\cite{hazle-raeburn-sims-webster}, who detailed in \cite[Theorems 4.4 and 4.5]{hazle-raeburn-sims-webster} the correspondence between higher-rank graphs on the one hand, and on the other hand, edge-colored directed graphs with an equivalence relation on their category of paths. 
 In this perspective, the factorization property of a $k$-graph is encoded in the set of ``commuting squares,'' or length-2 paths $ab \sim cd$ which are equivalent in the edge-colored directed graph.

The paper at hand constitutes a first step towards extending the geometric classification of graph $C^*$-algebras to the setting of higher-rank graphs.  Taking inspiration from 
\cite{drinen-thesis, bates-pask, crisp-gow, sorensen-first} as well as from \cite{errs}, we identify four moves (sink deletion, in-splitting, reduction, and delay) on row-finite, source-free $k$-graphs $\Lambda$ which preserve the Morita equivalence class of $C^*(\Lambda)$. These moves for $k$-graphs were inspired by their analogues for directed graphs, and therefore involve adding or removing edges and vertices in $\Lambda$.  Performing such a move on a $k$-graph affects the factorization property, though, as length-2 paths may become longer or shorter. Thus, geometric classification in the $k$-graph setting faces a new technical hurdle: one must identify how to adjust the factorization property after each move, so that the resulting object is still a $k$-graph.

As discussed in the introduction to \cite{bates-pask}, the moves of in-splitting and delay originate in  symbolic dynamics.  For shifts of finite type, the natural relations of conjugacy and flow equivalence \cite{parry-sullivan} are generated by matrix operations which, when translated into the graph setting, correspond to the moves of insplitting,  outsplitting and delay.  (See \cite[Sections 2.3 and 2.4]{lind-marcus} for more details.)  For directed graphs, the analogues {\tt (S)} of sink deletion and {\tt (R)} of reduction were first isolated by S\o{}rensen \cite{sorensen-first}, drawing on the very general framework given in \cite{crisp-gow} for modifying a directed graph without changing its Morita equivalence class.  The main result of \cite{sorensen-first} (Theorem 4.3) establishes that, for directed graphs $E, F$ with finitely many vertices such that $C^*(E)$ and $C^*(F)$ are simple, any stable isomorphism $C^*(E) \otimes \mathcal K \cong C^*(F) \otimes \mathcal K$ must arise from a finite sequence of insplittings, outsplittings, Cuntz splice,  the moves {\tt (S)} and {\tt (R)}, and their inverses.  As mentioned above, a series of papers by Eilers, S\o{}rensen, and others followed, which culminated in the complete classification of unital graph $C^*$-algebras in \cite{errs}.

We now outline the structure of the paper.  The  picture of higher-rank graphs as arising from edge-colored directed graphs  underlies our work in this paper, and so we take some pains in Section \ref{sec:Notation} to assure the reader of the equivalence between our approach to $k$-graphs and the more common category-theoretic perspective.  To obtain our Morita equivalence results, we rely heavily on a generalization of the gauge-invariant uniqueness theorem for $k$-graphs \cite{kp}, and on Allen's results \cite{allen} about corners in higher-rank graphs, so we also review these notions in  Section \ref{sec:Notation}. 

 Each of Sections \ref{sec:InSplit} through \ref{sec:Reduction}  is dedicated to one of our four Morita equivalence preserving moves on $k$-graphs.  For each move, we first ensure that its output is a $k$-graph, and then we show that the resulting $k$-graph $C^*$-algebra is Morita equivalent to our original $C^*$-algebra.
 We begin with in-splitting in Section \ref{sec:InSplit}. We first describe conditions under which we can ``in-split'' a $k$-graph at a vertex $v$ -- that is, create two copies of $v$ and divide the edges with range $v$ among the two copies -- in such a way that the resulting object is still a $k$-graph (Theorem \ref{thm:InsplitKG}).  Theorem \ref{thm:InsplitIso} then establishes 
  that insplitting produces a $C^*$-algebra which is isomorphic to our original one, not merely Morita equivalent.  Section \ref{sec:delay} studies the move of ``delaying'' an edge by breaking it into two edges. In order to delay an edge in a $k$-graph, the $k$-graph's factorization rule also forces us to delay many of the edges of the same color.  In Theorem \ref{thm:DelayKG}, we show that this move results in a $k$-graph.  Moreover, its $C^*$-algebra is Morita equivalent to that of our original $k$-graph (Theorem \ref{thm:DelayME}). 
In Section \ref{sec:Sink}, we  show in Theorem \ref{thm:SinkKG} that if a vertex is a sink -- that is, it emits no edges of a given color -- then after deleting the sink and all incident edges, we are still left with a $k$-graph.  The fact that this move
does not change the Morita equivalence class of the $k$-graph $C^*$-algebra is established in Theorem \ref{thm:SinkME}. 
 Finally, we turn to {\em reduction} in Section \ref{sec:Reduction}, where we identify when contraction of a ``complete edge'' (see Definition \ref{not:CompleteEdge}) in a $k$-graph produces a $k$-graph (Theorem \ref{thm:ReductionKG}).  In this case, the $C^*$-algebra of the resulting $k$-graph is always Morita equivalent  to the original $k$-graph $C^*$-algebra, by Theorem \ref{thm:ReductionME}. 
 
 Throughout the paper, we include examples showcasing the moves, and indicating the necessity of  our hypotheses.

{\bf Acknowledgments:} The authors thank Kenton Ke, Emily Morison, Jethro Thorne, and Ryan Wood for helpful comments. This research project was supported by NSF grant DMS-1800749 to E.G., and by the University of Montana's Small Grants Program.

\section{Notation} \label{sec:Notation}

Fix an integer $k \ge 1$. As our main objects of study in this paper are $k$-graphs (higher-rank graphs), we begin by recalling their definition.  First, however, we specify that throughout this paper we regard $0$ as an element of $\N$, and we view ${\mathbb N}^k$ as a category, with composition of morphisms given by addition.  Consequently, $\N^k$ has one object (namely $(0, \ldots, 0)$).  For $n = \sum_{i=1}^k n_i e_i \in {\mathbb N}^k,$ we write $|n| = \sum_i n_i$.

\begin{definition}\label{def:kgraphRB}\cite[Definitions 1.1]{kp}
Let $\Lambda$ be a countable category and $d: \Lambda \to \N^k$ a functor. If $(\Lambda, d)$ satisfies the {\em factorization rule} -- that is, for every morphism $\lambda \in \Lambda$ and $n,m \in \N^k$ such that $d(\lambda) = n + m$, there exist unique $\mu, \nu \in \Lambda$ such that $d(\mu) = m$, $d(\nu) = n$, and $\lambda = \mu \nu$ -- then $(\Lambda,d)$ is a \textit{k}-graph.

We write $\Lambda^1 = \{ \lambda \in \Lambda: |d(\lambda)|=1\}$ and $\Lambda^0 = d^{-1}(0)$. 
If $e \in \Lambda^1$, we say $e$ is an {\em edge} of $\Lambda$, and $\Lambda^0$ is the set of {\em vertices} of $\Lambda$.

Observe that the factorization rule guarantees, for each $\lambda \in \Lambda$, the existence of unique $v, w \in \Lambda^0$ such that $v \lambda w = \lambda;$ we will write $r(\lambda)$ for $v$ and $s(\lambda)$ for $w$.  Similarly, we write 
\[ v \Lambda= \{ \lambda \in \Lambda: r(\lambda) = v\} \quad \text{ and } \quad v\Lambda^n = \{ \lambda \in v \Lambda: d(\lambda) = n\}\]
for any $n\in \N^k$.  The sets $\Lambda w, \Lambda^n w$ are defined analogously.
\end{definition}

Our reason for the convention that the source of  a morphism in $\Lambda$ lies on its right, and its range lies to the left, arises from the Cuntz--Krieger relations used to define $k$-graph $C^*$-algebras; see Definition \ref{def:kgraph-algebra} and Remark \ref{rmk:kgraph-projs} below.

We now briefly describe how to model $k$-graphs using $k$-colored graphs as we use this framework extensively for our constructions.
Following \cite{hazle-raeburn-sims-webster} we let $G = (G^0, G^1, r, s)$ denote a directed graph with $G^0$ its set of vertices and $G^1$ its set of edges; $r, s: G^1 \to G^0$ are the range and source map respectively. For an integer $n \ge 2$ let $G^n$ denote the paths of length $n$ in $G$.  By a slight abuse of notation, if $\delta \in G^n$ we will write $|\delta| := n$.

We now color the graph $G$ by assigning to each edge one of the standard basis vectors, $e_i$, of ${\mathbb N}^k$ and let $G^{e_i}$ be the set of edges assigned to $e_i$, so that $G^1 = \bigcup_{i=1}^k G^{e_i}$.  The {\em path category}, $G^* = \bigcup_{n\in \N} G^n$, may now be equipped with a {\em degree functor} $d: G^* \to {\mathbb N}^k$, given on the vertices by $d(v) = 0$ for all $v \in G^0$, and on the edges by $d(f) = e_i$ if $f$ was assigned the basis vector $e_i$.  On longer paths, $d$ is extended to be additive: $d(f_n \cdots f_1) = \sum_{i=1}^n d(f_i)$. (Our reason for this enumeration of the edges in a path is that in a $k$-graph, by Definition \ref{def:kgraphRB}, we have $s(f_i) = r(f_{i-1})$ whenever $f_i f_{i-1}$ is a well-defined product of morphisms in a $k$-graph.  Consistency with this definition requires that the right-most edge in a path, $f_1$, denotes the path's initial edge and the left-most edge, $f_n$, its final edge.) As usual, the range and source maps $r,s: G^1 \to G^0$ extend to well-defined maps from $G^*$ to $G^0$, which we  continue to denote by $r$ and $s$.

\goodbreak

Theorem 4.5 of \cite{hazle-raeburn-sims-webster} establishes that, if $G$ is a $k$-colored graph as described above, then $\Lambda = G^* / \sim$ is a $k$-graph for any $(r,s,d)$-preserving equivalence relation $\sim$ on $G^*$  which also satisfies 
\begin{enumerate}
    \item[(KG0)] If $\lambda \in G^*$ is a path such that $\lambda = \lambda_2 \lambda_1$, then $[\lambda] = [p_2 p_1]$ whenever $p_1 \in [\lambda_1]$ and $p_2 \in [\lambda_2]$.
    \item[(KG1)] If $f,g \in G^1$ are edges, then $f \sim g \Leftrightarrow f = g$.
    \item[(KG2)] \textbf{Completeness:} For every $\mu = \mu_2 \mu_1 \in G^2$ such that $d(\mu_1) = e_i$, $d(\mu_2) = e_j$, there exists a unique $\nu = \nu_2 \nu_1 \in G^2$ such that $d(\nu_1) = e_j$, $d(\nu_2) = e_i$ and $\mu \sim \nu$. 
    \item[(KG3)] \textbf{Associativity:} For any $e_i$-$e_j$-$e_\ell$ path $abc \in G^3$ with $i,j,\ell$ all distinct, the  $e_\ell$-$e_j$-$e_i$  paths $hjg, nrq$ constructed via the following two routes are equal. 
    \end{enumerate}

\begin{minipage}{0.45\textwidth}
      \begin{enumerate}
        \item[Route 1:]
        Let $ab \sim de$, so $abc \sim dec$.\\
        Let $ec \sim fg$, so $abc \sim dfg$.\\  
        Let $df \sim hj$, so $abc \sim hjg$.\\
        \item[Route 2:] Let $bc \sim km$, so $abc \sim akm$.\\
        Let $ak \sim np$, so $abc \sim npm$.\\
        Let $pm \sim rq$, so $abc \sim nrq$.
    \end{enumerate}
    \end{minipage}
    \begin{minipage}{0.45\textwidth}
\begin{figure}[H]
\begin{tikzcd}[column sep=0.7cm,row sep=0.7cm]
	& \bullet \ar[rr, red, "d"]  & & \bullet \\
	\bullet  \ar[ru, blue, dashed, "f"] \ar[rr, red, near end, "j=r"] & & \bullet \ar[ru, blue, dashed, "h=n"]  & \\
	& \bullet \ar[rr, red, "b" near start] \ar[uu,black,dotted, "e" near start]& & \bullet \ar[uu,black,dotted, "a"] \\
	\bullet \ar[ru, blue, dashed, "c"] \ar[rr, red, "m"] \ar[uu,black,dotted, "g=q"] & & \bullet \ar[ru, blue, dashed, "k"] \ar[uu,black,dotted,"p", near start]& \\
\end{tikzcd}
\caption{(KG3)}
\end{figure}
\end{minipage}

\noindent
In fact, \cite[Theorem 4.4]{hazle-raeburn-sims-webster} shows that every $k$-graph arises in this way.  That is, given a $k$-graph $\Lambda$, we obtain a directed graph $G$ by setting $G^0 = \Lambda^0,  G^1 = \Lambda^1$. (This justifies our decision to call $\Lambda^0$ the vertices of $\Lambda$ and $\Lambda^1$ the edges of $\Lambda$.) Transferring the degree map $d: \Lambda \to \N^k$ to $G$ makes $G$ a $k$-colored graph; we obtain an equivalence relation on $G^*$ by setting $\lambda \sim \mu$ if the paths $\lambda, \mu$ represent the same morphism in $\Lambda$. The factorization rule in $\Lambda$ then implies that $\sim$ satisfies (KG0) - (KG3).

In this paper, we fully exploit the equivalence between $k$-colored directed graphs with equivalence relations on the one hand, and $k$-graphs on the other hand. Our general strategy will be to define a move 
$M$ on a $k$-graph $\Lambda$ in terms of its impact on the 1-skeleton $G$ and the equivalence relation $\sim$ which give  rise to  $\Lambda$.  This produces a new colored graph $G_M$ with  a new equivalence relation $\sim_M$, which we then show satisfies (KG0) - (KG3) so that the quotient $G_M/\sim_M$ is a new $k$-graph $\Lambda_M$.

For $\lambda \in G^*$ we notate its equivalence class under $\sim$ as $[\lambda] \in \Lambda$.
For $n \in \N^k$ we write  $$\Lambda^n = \{ [\lambda] \in \Lambda : d([\lambda]) = n \}.$$ 

For our purposes in this paper, we will also need an alternative characterization of the equivalence relations on $G^*$ which give rise to $k$-graphs. We begin by observing that an inductive application of the factorization rule of Definition \ref{def:kgraphRB} reveals that if $\Lambda$ is a $k$-graph, then for any morphism $\lambda \in \Lambda$ and ordered $n$-tuple $(m_1, \dots, m_n)$ of elements of $\mathbb{N}^k$ such that $|m_i| = 1$ for all $i$ and  $m_1 + \cdots + m_n = d(\lambda)$, there is a unique set of edges $\lambda_1, \dots, \lambda_n \in \Lambda^1$ such that $\lambda = \lambda_n \cdots \lambda_1$ where $d(\lambda_i) = m_i$.

\begin{definition}\label{not:ColorOrder}
\label{not:PathSteps}
For a finite path $\lambda$ in an edge-colored directed  graph $G$, 
 let $\lambda_i$ denote the $i$th edge of $\lambda$ (counting from the source of $\lambda$). The {\em color order} of $\lambda $ is the $|\lambda|$-tuple $(d(\lambda_1), d(\lambda_2),$ $\dots , d(\lambda_{|\lambda|}))$.
\end{definition}

This leads us to the following condition on an equivalence relation $\sim$ on $G^*$:

\begin{enumerate} 
    \item[(KG4)] For each $\lambda \in G^*$ and each permutation of the color order of $\lambda$, there is a unique path $\mu \in [\lambda]$ with the permuted color order. 
\end{enumerate}

\begin{theorem}\label{thm:KGalternative}
Let $G$ be an edge-colored directed graph and suppose $\sim$ is an $(r, s, d)$-preserving equivalence relation on $G^*$ satisfying (KG0).  The relation $\sim$ satisfies (KG1), (KG2), and (KG3) (and hence $G^*/\sim$ is a $k$-graph) if and only if $\sim$ satisfies (KG4).  
\label{thm:kgraphAlt}
\end{theorem}

\begin{proof}
	First, assume (KG0) and (KG4) hold for $\sim$ and consider an $e_i$-$e_j$-$e_\ell$ path $\lambda \in G^3$. Convert $\lambda$ into two $e_\ell$-$e_j$-$e_i$ paths via the routes described in (KG3) and label them $\mu$ and $\nu$. Since $\mu \sim \lambda$ and $\nu \sim \lambda$ by construction, the fact that $\sim$ is an equivalence relation implies   that $\mu \sim \nu$. Condition (KG4) and the fact that $\mu, \nu$ have the same color order now gives  $\mu$ = $\nu$. Thus (KG3) holds. Similarly, if $\lambda \in G^2$, then there exists a unique $\mu \in [\lambda]$ of each permuted color order. Thus (KG2) holds. Finally for $e,f \in G^1$ we have $e \sim f \implies d(e) = d(f) \implies e = f$, since each color order has a unique associated path. Also $e = f \implies e \sim f$. Thus (KG1) holds.

	Now assume $\sim$ satisfies (KG0), (KG1), (KG2), and (KG3).  We know from \cite[Theorem 4.5]{hazle-raeburn-sims-webster}    that $\Lambda := G^* / \sim$ is a $k$-graph. Thus, fix $\delta \in G^*$, and choose a sequence of basis vectors $(m_j)_{j=1}^{|\delta|},$ with $m_j \in \{ e_i\}_{i=1}^k$ for all $j$, such that  $d(\delta) = \sum_{j = 1}^{|\delta|}m_{j}$.  An inductive application of the factorization rule of Definition \ref{def:kgraphRB} implies the existence of a unique path $\gamma = \gamma_{|\delta|} \cdots \gamma_2 \gamma_1 \in [\delta]$ where $d(\gamma_j) = m_j$ for every $j$. Since our ordering of the basis vectors $(m_1, \ldots, m_{|d(\lambda)|})$ was arbitrary, it follows that $\sim$ satisfies (KG4).
\end{proof}

\begin{notation}\label{not:Source}
A $k$-graph $\Lambda$ is {\em row-finite} if for all $v \in \Lambda^0$ and all $1 \leq i \leq k$, we have $| \{ \lambda \in \Lambda^{e_i}: r(\lambda) = v\}| < \infty.$
We say $v \in \Lambda^0$ is a {\em source} if there is $i$ such that $r^{-1}(v) \cap \Lambda^{e_{i}} = \emptyset$.
\end{notation}
In this paper we will focus exclusively on row-finite source-free $k$-graphs.

\begin{definition}\cite[Definition 1.5]{kp}, \cite[Definition 7.4]{kps-hlogy}
\label{def:kgraph-algebra}
Let $\Lambda$ be a row-finite, source-free $k$-graph $\Lambda$. 
A {\em Cuntz--Krieger $\Lambda$-family} 
is a collection of 
projections $\{P_v: v \in \Lambda^0\}$ and partial isometries $\{ T_f: f \in \Lambda^1\}$ satisfying the {\em Cuntz--Krieger relations:}
\begin{enumerate}
    \item[(CK1)] The projections $P_v$ are mutually orthogonal.
    \item[(CK2)] If $a, b, f, g \in \Lambda^1$ satisfy $af \sim gb $, then $T_a T_f = T_g T_b.$
    \item[(CK3)] For any $f \in \Lambda^1$ we have $T_f^* T_f = P_{s(f)}.$
    \item[(CK4)] For any $v\in \Lambda^0$ and any $1 \leq i \leq k,$ we have $\displaystyle P_v = \sum_{f: r(f) = v, d(f) = e_i}T_fT_f^*.$
\end{enumerate}
There is a universal $C^*$-algebra for these generators and relations, which is denoted $C^*(\Lambda) = C^*(\{ p_v, t_f\})$.  For any Cuntz--Krieger $\Lambda$-family $\{ P_v, T_f\}$, we consequently have a  surjective $*$-homomorphism $\pi: C^*(\Lambda) \to C^*(\{ P_v, T_e\})$, such that $\pi(p_v) = P_v$ and $\pi(t_f) = T_f$ for all $v \in \Lambda^0, f \in \Lambda^1$.
\end{definition}

\begin{remark}
\label{rmk:kgraph-projs}
Observe that if $\{ T_f, P_v\}$ is a Cuntz--Krieger $\Lambda$-family, then (CK3) implies that $T_f P_{s(f)} = T_f$.  Similarly, by (CK4) and the fact that a sum of projections is a projection iff those projections are orthogonal, $P_{r(f)} T_f = T_f$.
Thus, viewing edges in $\Lambda^1$ as pointing from right to left ensures the compatibility of concatenation of edges in $\Lambda$ with the multiplication in $C^*(\Lambda)$.
\end{remark}

If $\Lambda = G^*/\sim$, and  $\mu, \nu \in  G^*$ represent the same equivalence class in $\Lambda$, then Condition (CK2), together with conditions (KG0)--(KG2), guarantees that 
\[ t_{\mu_n} \cdots t_{\mu_1} = t_{\nu_n} \cdots t_{\nu_2}  t_{\nu_1}.\]
  Thus, for $[\mu] \in \Lambda$, we define $t_\mu := t_{\mu_n} \cdots t_{\mu_1}$. 
\cite[Lemma 3.1]{kp} then implies that
$ \{ t_{\mu} t_{\nu}^* : [\mu], [\nu] \in \Lambda\} $
densely spans $C^*(\Lambda).$

\begin{remark}
We have opted to describe $C^*(\Lambda)$ purely in terms of the partial isometries associated to the vertices and edges, rather than the more common description using all of the partial isometries $\{ t_\lambda: \lambda \in \Lambda\},$ because all of our ``moves'' on $k$-graphs occur at the level of the edges.  
\end{remark}

A crucial ingredient in our proofs that all of our moves preserve the Morita equivalence class of $C^*(\Lambda)$ is the {\em gauge-invariant uniqueness theorem}.  To state this theorem, observe first that the universality of $C^*(\Lambda)$ implies  the existence of a canonical action $\alpha$ of ${\mathbb T}^k$ on $C^*(\Lambda)$ which satisfies 
\[ \alpha_z(t_e) = z^{d(e)} t_e\qquad \text{ and } \qquad \alpha_z(p_v) = p_v\]
for all $z \in {\mathbb T}^k, e \in \Lambda^1$ and $ v \in \Lambda^0$.

\begin{theorem}
\label{thm:GIUT}
\cite[Theorem 3.4]{kp}
Fix a row-finite source-free $k$-graph $\Lambda$ and a $*$-homomorphism $\pi :  C^*(\Lambda) \to B$.  If $B$ admits an action $\beta$ of ${\mathbb T}^k$ such that $\pi \circ \alpha_z = \beta_z \circ \pi$ for all $z \in {\mathbb T}^k$,  and for all $v \in \Lambda^0$ we have $\pi(p_v) \not= 0$, then $\pi$ is injective.
\end{theorem}

Many of the actions $\beta$ that will appear in our applications of the gauge-invariant uniqueness theorem take the form described in the following Lemma.   The proof is a standard   argument, using the universal property of $C^*(\Lambda)$ to establish that $\beta_z$ is an automorphism for all $z$, and using an $\epsilon/3$ argument to show that $\beta$ is strongly continuous, so we omit  the details.

\begin{lemma}\label{lemma:Action}
Let $(\Lambda , d)$ be a row-finite  source-free $k$-graph. Given a functor $R: \Lambda \to \mathbb{Z}^k$, the function $\beta: \mathbb{T}^{k} \to \textup{Aut}(C^*(\Lambda))$ which satisfies
\[ \beta_z(t_{\mu}t_{\nu}^*) = z^{R(\mu) - R(\nu)} t_{\mu}t_{\nu}^*\]
for all $\mu, \nu \in \Lambda$ and $z \in \mathbb T^k$,
is an action of ${\mathbb T}^k$ on $C^*(\Lambda)$.
\end{lemma}

In particular, we can apply the above Lemma whenever we have a function $R: \Lambda^1 \to \mathbb Z^k$ such that, if we extend $R$ to a function on $G^n$ by the formula 
 \[R(\lambda_n \cdots \lambda_1) := R(\lambda_n) + \cdots + R(\lambda_1),\]
   $R$ becomes a well-defined function on $\Lambda$.

In addition to Theorem \ref{thm:GIUT} and Lemma \ref{lemma:Action}, our proofs that delay and reduction preserve Morita equivalence will rely on Allen's description \cite{allen} of corners in $k$-graph $C^*$-algebras.  To state Allen's result, we need the following definition.
\begin{definition}
\label{def:saturation}
Let $(\Lambda, d)$ be a row-finite $k$-graph.
The {\em saturation} $\Sigma(X)$ of a set $X\subseteq \Lambda^0$ of vertices is the smallest set $S\subseteq \Lambda^0$ which contains $X$ and satisfies
\begin{enumerate}
\item (Heredity) If $v \in S$ and $r(\lambda) = v$ then $s(\lambda) \in S$;
\item (Saturation) If $s(v\Lambda^n) \subseteq S$ for some $n \in \N^k$ then $v \in S$. 
\end{enumerate}
\end{definition}
The following Theorem results from combining Remarks 3.2(2), Corollary 3.7, and Proposition 4.2 from \cite{allen}.
\begin{theorem}\cite{allen}
\label{thm:allen}
Let $(\Lambda, d)$ be a $k$-graph and $X \subseteq \Lambda^0$.  Define 
\[ P_X = \sum_{v\in X} p_v \in M(C^*(\Lambda)).\]
If   $\Sigma(X) = \Lambda^0$, then $P_X C^*(\Lambda) P_X$ is Morita equivalent to $C^*(\Lambda).$  
\end{theorem}

\section{In-splitting}
\label{sec:InSplit}

The move of {\em in-splitting} a $k$-graph at a vertex $v$ which we describe in this section should be viewed as the analogue of the {\em out-splitting} for directed graphs which was introduced by Bates and Pask in \cite{bates-pask}.  This is because the Cuntz--Krieger conditions used by Bates and Pask to describe the  $C^*$-algebra of a directed graph differ from the standard Cuntz--Krieger conditions for $k$-graphs.  In the former, the source projection $t_e^* t_e$ of each partial isometry $t_e$, for $ e \in \Lambda^1,$ is required to equal $p_{r(e)}$, whereas our Definition \ref{def:kgraph-algebra} requires $t_e^* t_e = p_{s(e)}.$

The following definition indicates the care that must be taken in in-splitting for higher-rank graphs. The pairing condition of Definition \ref{def:pairing} is necessary even for 2-graphs (cf.~Examples \ref{ex:Insplit} below), but is vacuous for directed graphs. Although in- and out-splitting for directed graphs (cf.~\cite{bates-pask, sorensen-first}) allow one to ``split'' a vertex into any finite number of new vertices,  the delicacy of the pairing condition has led us to  ``split'' a vertex into only two new vertices.

\begin{definition}\label{def:pairing}
Let $(\Lambda, d)$ be a source-free $k$-graph with  1-skeleton $G = (\Lambda^{0}, \Lambda^1, r, s)$ and path category $G^*$. Fix  $v \in \Lambda^{0}$. Partition $r^{-1}(v) \cap \Lambda^1$ into two non-empty sets $\E_1$ and $\E_2$ satisfying the \textit{pairing condition}: if $a, f \in r^{-1}(v) \cap \Lambda^1$ and there exist edges $g ,b\in \Lambda^1$ such that $ ag  \sim fb  ,$ then $f$ and $a$ are contained in the same set. 
\end{definition}
We will use the partition $\mathcal E_1 \cup \mathcal E_2$ of $r^{-1}(v) \cap \Lambda^1$ when we in-split $\Lambda$ at $v$ in Definition \ref{def:Insplit} below.  First, however, we pause to examine some consequences of the pairing condition.

\begin{remark}
If $a \not= f$ are edges of the same color, then  the relation $\sim$ underlying the $k$-graph $\Lambda$ will never satisfy $a g \sim f b$.  Thus, the pairing condition places no restrictions on edges of the same color.  It follows that our definition of insplitting (Definition \ref{def:Insplit} below) is consistent with the definition of insplitting \cite[Section 5]{bates-pask} for directed graphs.

However, for $k > 2$, the pairing condition means that not all $k$-graphs can be in-split at all vertices.  Satisfying the pairing condition requires that if $f b \sim a g $ then $f,a$ are in the same set.  This may force one of the sets $\E_i$ to be empty, which is not allowed under Definition \ref{def:pairing}.
\end{remark}

\goodbreak

\begin{examples}
{\color{white},}

\label{ex:Insplit}
\begin{enumerate}
\item The property of having a valid partition $\mathcal E_1 \cup \mathcal E_2$ of $r^{-1}(v) \cap \Lambda^1$ at a given vertex  $v$ depends not  only on the 1-skeleton $G$ of $\Lambda$, but also  on the equivalence relation $\sim$ giving $\Lambda = G^*/\sim$.  For example, let $\Lambda$ be a $2$-graph with one vertex, $\Lambda^{e_1} = \{ a,b\}$ and $\Lambda^{e_2} = \{ e,f \}$.
\begin{enumerate}
\item If we define $ae\sim ea, af\sim fa, be\sim eb, bf\sim fb$,  repeatedly applying the pairing condition gives $\E_1 = \{a,b,e,f \}$ and so no valid partition is possible.  Thus $\Lambda$ cannot be in-split.
\item On the other hand, if we set $ae\sim ea, af\sim eb, be\sim fa, bf\sim fb$, we can take $\E_1 = \{ a,e \}$ and $\E_2 = \{b,f\}$.  Thus in this case $\Lambda$ can be in-split.
\end{enumerate}
\item It may be possible to find two different valid partitions at a given vertex. Let $\Gamma$ be a $2$-graph with one vertex, $\Gamma^{e_1} = \{ a,b,c,d \}$
and $\Gamma^{e_2} = \{ e,f,g,h \}$, and  the equivalence relation
\begin{align*}
ae\sim ea, \ af\sim eb, \ ag \sim  ec, \ ah \sim ed, & \ be\sim fa, \ bf\sim fb, \ bg \sim fc, \ bh\sim fd , \\
ce \sim ga, \ cf \sim  gb,\  cg \sim gc , \ ch \sim gd , \ & de\sim hd,\  df\sim hc, \ dg \sim  hb, \ dh \sim  ha .
\end{align*}

\noindent Then $\E_1 = \{ a, c, e , g \}$, $\E_2 = \{ b,f , d , h\}$ and $\E_1 = \{ a, e  \}$, $\E_2 = \{ b,c,d,f,g,h\}$ are two  partitions satisfying the pairing condition.
\end{enumerate}
\end{examples}

\begin{lemma}\label{lemma:Insplit}
For $j \in \{1,2\}$, $\E_j$ has an edge of every color.
\end{lemma}

\begin{proof}
Note that there exists $e \in \E_j$ and $s(e)$ is not a source. Thus for $1 \leq i \leq k$ there exists $f \in r^{-1}(s(e)) \cap \Lambda^{e_i}$, and hence there exists a unique $\lambda = \lambda_1 \lambda_2 \in G^2$ such that $d(\lambda_2) = d(e)$, $d(\lambda_1) = e_i$, and $\lambda \sim e f$. Therefore, by the definition of $\E_j$, we have $\lambda_1 \in \E_j$. 
\end{proof}

\begin{definition}\label{def:Insplit}
Let $(\Lambda, d)$ be a source-free $k$-graph. Fix $v \in \Lambda^0$ and a partition $\mathcal E_1 \cup \mathcal E_2$ of $r^{-1}(v) \cap \Lambda^1$ satisfying Definition \ref{def:pairing}.  We define the associated directed $k$-colored graph $G_{I} = (\Lambda^{0}_I,\Lambda^1_I,r_I,s_I)$ with degree map $d_I$ by
\begin{align*}
\Lambda_I^0 = (\Lambda^{0} \setminus \{v\})\cup \{v^1, v^2\}& \quad
\Lambda_I^1 = (\Lambda^1 \setminus s^{-1}(v))\cup\{f^1, f^2~|~f \in \Lambda^1,~s(f) = v\} \text{, with } \\
d_I (g) &= d(g) \text{ for } g \in \Lambda^1 \setminus s^{-1}(v) \text{ and } d_I(f_i) = d(f) .
\end{align*}

\noindent
The range and source maps in $G_I$ are defined as follows:
\begin{align*}
\text{For } &f \in \Lambda^1 \text{ such that } s(f)\neq v ,~ s_I(f)=s(f)  , \\
\text{for }  &f \in \Lambda^1 \text{ such that } r(f) \neq v, ~ r_I(f)=r(f) \text{ and } r_I(f^i) =r(f), \\
\text{for }  &f\in \Lambda^1 \text{ such that } s(f)=v , ~ s_I(f^i) = v_i \text{ for } i = 1,2 \\
\text{for } &f  \in \Lambda^1 \text{ such that } r(f)=v \text{ and } f \in \E_i , ~ r_I(f)=v^i.
\end{align*}
\end{definition}

\begin{examples}
{\color{white},}

\label{ex:insplit-1}
\begin{enumerate}
\item The graph $G$ shown below admits a unique equivalence relation $\sim$ such that $G^*/\sim$ is a 2-graph $\Lambda$, because there is always at most one red-blue path (and the same number of blue-red paths) between any two vertices.  We may in-split at the vertex $v$ with $\E_1 = \{ a,e \}$ and $\E_2 = \{ b,f\}$. We duplicate $x \in s^{-1} (v)$ to $x^1 , x^2$ with sources $v^1$ and $v^2$ respectively, and $r(x^i) = r(x)$ for each $i$.
\begin{figure}[H]
    \begin{tikzcd}[column sep=0.5cm,row sep=0.3cm]
    G \\
     &&& \\
    \bullet \ar[loop above, blue, dashed] \ar[loop left, red] \ar[dd, blue, dashed] \ar[dr, red] & &  \\
     & \bullet \ar[loop above, blue, dashed] \ar[r, red,"p"] \ar[dd, blue, dashed, "a"] & \bullet \ar[loop above, blue, dashed] \ar[dd, blue, dashed, "q"] \\
    \bullet \ar[loop left, red] \ar[dr, red,  "e" ] & &  \\
    & v \ar[r, red,"x"] & \bullet \\
    \bullet \ar[loop left, red] \ar[ur, red, "f" ] & &  \\
    & \bullet \ar[loop below, blue, dashed] \ar[r, red] \ar[uu, blue, dashed, "b"] & \bullet \ar[loop below, blue, dashed ] \ar[uu, blue, dashed ]  \\
    \bullet \ar[loop below, blue, dashed] \ar[loop left, red] \ar[uu, blue, dashed] \ar[ur, red] & & 
    \end{tikzcd}
    \hspace{1.5cm}
 \begin{tikzcd}[column sep=0.5cm,row sep=0.3cm]
 G_I \\
   &&& \\
    \bullet \ar[loop above, blue, dashed ] \ar[loop left, red] \ar[dd, blue, dashed] \ar[r, red] & \bullet \ar[loop above, blue,dashed ] \ar[dr, red,"p"] \ar[dd, blue,dashed,"a"] &  \\
     &  & \bullet \ar[loop above, blue, dashed] \ar[dd, blue, dashed, "q"]  \\
    \bullet \ar[loop left, red] \ar[r, red,"e"] & v^1 \ar[dr, red,"x^1"] &  \\
    &  & \bullet \\
    \bullet \ar[loop left, red] \ar[r, red,"f"] & v^2 \ar[ur, red,"x^2"] &  \\
    & & \bullet \ar[loop below, blue, dashed] \ar[uu, blue, dashed]  \\
    \bullet \ar[loop below, blue,dashed] \ar[loop left, red] \ar[uu, blue, dashed] \ar[r, red] & \bullet \ar[loop below, blue, dashed] \ar[uu, dashed, blue,"b"] \ar[ur, red] & 
\end{tikzcd}
\caption{First example of in-splitting}
\label{fig:isplt1}
\end{figure}
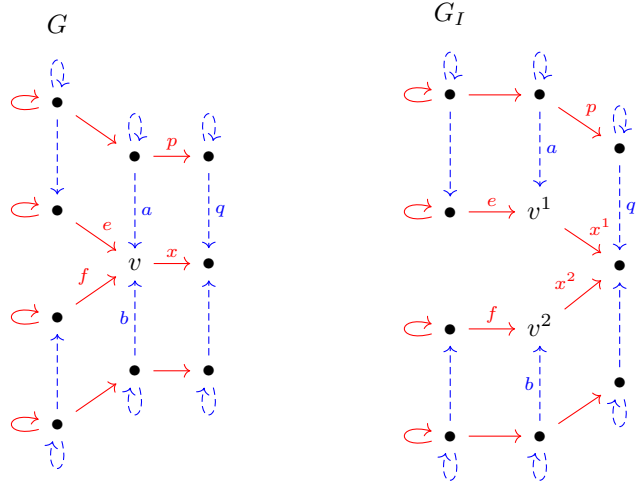

\noindent

\item We now give an example of an in-splitting where the vertex at which the splitting occurs has a loop.
The graph $G$  in Figure \ref{fig:isplt2} gives rise to multiple 2-graphs; we fix the $2$-graph structure on $G$ given by the equivalence relation
\begin{equation*}
    \begin{split}
        & a e \sim e a ,~ ce \sim fa,   \ g c \sim b f ,~ b g \sim g b .
    \end{split}
\end{equation*}
Thus, the sets $\mathcal E_1 = \{ c, f\}, \mathcal E_2 = \{ b, g\}$ satisfy the pairing condition, and we can in-split at $v$.

\begin{figure}[H]
\label{fig:isplt2}
\begin{tikzcd}[column sep=1cm,row sep=.5cm]
G \\ 
\ar[out=135, in=-135, distance=4em, red, "e"'] \ar[out=150, in=-150, distance=2em, dashed, blue, "a"] \bullet \ar[r, bend left, red, "f"] \ar[r, bend right, dashed, blue, "c"']
& v \ar[out=45, in=-45, distance=4em, red, "g"] \ar[out=30, in=-30, distance=2em, dashed, blue, "b"']
\end{tikzcd}
\hspace{1.5cm}
\begin{tikzcd}[column sep=1cm,row sep=.5cm]
G_I \\
\ar[out=135, in=-135, distance=4em,red, "e"'] \ar[out=150, in=-150, distance=2em, dashed, blue,, "a"] \bullet \ar[r, bend left, red, "f"] \ar[r, bend right, dashed, blue, "c"]
& v^1 \ar[r, bend left, red, "g^1"] \ar[r, bend right, dashed, blue, "b^1"]
& v^2 \ar[out=45, in=-45, distance=4em, red, "g^2"] \ar[out=30, in=-30, distance=2em, dashed, blue, "b^2"']
\end{tikzcd}
\caption{In-splitting at a vertex $v$ which has loops.}
\end{figure}
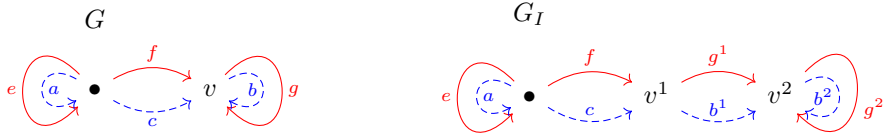

\noindent

\end{enumerate}
\end{examples}

\begin{remark}
While the vertex $v$ at which we in-split the graph $G$ from Figure \ref{fig:isplt1}
is a sink, and hence could also be handled by the methods of Section \ref{sec:Sink} below, one could easily modify $\Lambda$ to be sink-free (at the cost of a more messy 1-skeleton diagram)  without changing the essential structure of the in-splitting at $v$.
\end{remark}

In order to describe the factorization on $G_I$ which will make it a $k$-graph, we first introduce some notation.

\begin{definition}
\label{def:par-insplit}
Define a function $par: G_{I}^* \to G^*$ by
\begin{equation*}
\begin{split}
&  par(w) = w \text{ for all }w  \in \Lambda^0\backslash \{ v\}\text{ and }par(v^i) = v \text{ for }i=1,2,\\
& par(f^{i}) = f, \text{ for } f^i \in \{ f^1, f^2 | \; f \in \Lambda^1,  s(f) = v\}   \\
& par(f) = f, \text{ for } f \in \Lambda_{I}^1 \  \backslash \,   \{ f_1, f_2 | \; f \in \Lambda^1,  s(f) = v\}  \\
& par(\lambda) = par(\lambda_1) \cdots par(\lambda_n), \text{ for } \lambda = \lambda_1\cdots \lambda_n \in G_{I}^*.
\end{split}
\end{equation*}
\end{definition}

The effect of the function $par$ is to remove the superscript on any edge (or path) of $G_I$, returning its ``parent'' in $G$ (or $G^*$).
\begin{definition}
\label{def:insplit-equiv}
We define an equivalence relation on $G_{I}^*$ by $\lambda \sim_I \mu$ if and only if $par(\lambda) \sim par(\mu)$, $r_I(\lambda) = r_I(\mu)$, and $s_I(\lambda) = s_I(\mu)$. Define $\Lambda_{I} := G_{I}^*/\sim_I$; we say that $\Lambda_I$ is the result of {\em in-splitting} $\Lambda$ at $v$.
\end{definition}

\begin{examples}
{\color{white},}

\begin{enumerate}
\item Consider again the directed colored graph of Example \ref{ex:insplit-1}(1).  Observe that $x^1 a \sim_I qp$ in $G_I$ since both paths have the same range and source in $G_I$ and 
$[par (x^1 a )] = [xa]=[qp] = [par(qp)]$ in $\Lambda = G^*/\sim$.
\item In the directed colored graph of Example \ref{ex:insplit-1}(2), we have  $b^2 g^1\sim_I g^2 b^1 $ as both  paths have the same source and range and 
$[ par ( b^2 g^1 ) ] = [bg] = [gb] = [par ( g^2 b^1 )  ]$ in $\Lambda$.  Observe that although $G$ admitted multiple factorizations,  $G_I$ admits only this one.
\end{enumerate}

\end{examples}

\begin{remark}
\label{rmk:parent}
{\color{white},}

\begin{enumerate}
\item For any $\lambda \in G^*_I,$ if $s(\lambda) = s(\mu)$ and $par(\lambda) = par(\mu),$ we have  $\lambda = \mu.$ To see this, observe that by definition an edge  $e \in \Lambda^1$ satisfies $e = par(\mu)$ for at most two edges $\mu \in \Lambda_I^1.$  If $par(\mu) = par(\nu)$ and $\nu \not= \mu$, then without loss of generality we may assume $s(\mu) = v_1,  s(\nu) = v_2.$ Consequently, for $\eta \in \Lambda_I^1,$ at most one of the paths $\mu \eta, \nu \eta$ is in $G^*_I$, according to whether $\eta \in \E_1$ or $\eta \in \E_2.$ This implies our assertion.
\item Similarly, for any path $\lambda \in G^*,$ we have $\lambda = par(\mu)$ for at least one path $\mu \in G^*_I.$
\end{enumerate}
\end{remark}

\begin{theorem}\label{thm:InsplitKG}
If $(\Lambda, d)$ is a source-free $k$-graph, then the result $(\Lambda_{I}, d_I)$ of in-splitting $
\Lambda$  at a vertex $v$ is also a source-free $k$-graph.
\end{theorem}

\begin{proof}
Let $(\Lambda, d)$ be a source-free $k$-graph and let $(\Lambda_{I}, d_{I})$ be produced by in-splitting at some $v \in \Lambda^{0}$. First note that $\Lambda_I$ satisfies (KG0)  by our definition of $par$ and the fact that $\Lambda$ has the factorization property. Lemma \ref{lemma:Insplit} and our hypothesis  that  $\Lambda$ be source-free guarantee that all vertices in $\Lambda_I^0$ receive edges of all colors, so $\Lambda_I$ is source-free. Consider some path $\lambda \in G_I^*$ with color order $(m_1, \dots, m_n)$. Note that $par(\lambda)$ also has color order $(m_1, \dots, m_n)$, and since $\Lambda$ is a $k$-graph, for any permutation $(c_1, \dots, c_n)$ of $(m_1, \dots, m_n)$, there exists a unique $\mu' \in \Lambda$  that has color order $(c_{1}, \dots , c_n)$ and $\mu' \in [par(\lambda)]$. By Remark \ref{rmk:parent}, there exists a unique path $\mu \in \Lambda_I$ such that $par(\mu) = \mu'$ and $s(\mu) = s(\lambda)$.  By construction, $\mu$ has color order $(c_{1}, \dots , c_n)$ and $\mu \in [\lambda]_I$. Thus,  $[\lambda]_I$ contains a unique path for each permutation of the color order of $\lambda$, and so (KG4) is satisfied. Therefore, by Theorem \ref{thm:KGalternative}, $\Lambda_{I}$ is a $k$-graph.
\end{proof}

\begin{theorem}
\label{thm:InsplitIso}
Let $(\Lambda, d)$ be a row-finite, source-free \textit{k}-graph, and let $(\Lambda_I, d_I)$ be the in-split graph of $\Lambda$ at the vertex $v$ for the partition $\E_1 \cup\E_2$ of $r^{-1}(v) \cap \Lambda^1$. We have $C^*(\Lambda_I)\cong C^*(\Lambda)$.
\end{theorem}

\begin{proof}
Let $\{s_\lambda : \lambda \in \Lambda^0_I \cup \Lambda^1_I \}$ be the canonical Cuntz-Krieger $\Lambda_I$-family which generates $C^*(\Lambda_I)$. For $\lambda\in\Lambda^0\cup\Lambda^1$, define
\begin{equation*}
    T_{\lambda} = \sum_{par(e) = \lambda} s_{e}.
\end{equation*}

We first prove that $\{T_\lambda : \lambda \in \Lambda^0 \cup \Lambda^1\}$ is a Cuntz-Krieger $\Lambda$-family in $C^*(\Lambda_I)$.  Note that the set $\{T_\lambda : \lambda \in \Lambda^0\}$ is a collection of non-zero mutually orthogonal projections since each $T_\lambda$ is a sum of projections satisfying the same properties. Therefore  $\{T_\lambda : \lambda \in \Lambda^0\}$ satisfies (CK1). Now, choose $ab, cd \in G^2$ such that $[ab] = [cd]$. As in Remark \ref{rmk:parent},
observe that the sum defining $T_a$ contains at most two elements, and the only way it will contain two elements is if $s(a) = v$.  In that case, if $ab \in G^2,$ then either $b \in \E_1$ or $b \in \E_2$, so if $f \in G_I^1$ satisfies $par(f) = b$ then $r(f) \in \{ v^1, v^2\}$, and so 
there is only one path $ef \in G_I^2$ whose parent is $ab$.  Making the same argument for the paths in $G^2_I$ with parent $cd$ and using the factorization rule in $\Lambda_I$, we obtain
\begin{equation*}
T_aT_b = \sum_{par(e) = a}s_{e} \sum_{par(f)= b}s_{f} = \sum_{par(ef) = ab} s_e s_f = \sum_{par(gh) = cd}s_g s_h = T_c T_d.
\end{equation*}
Thus $\{T_\lambda : \lambda \in \Lambda^0 \cup \Lambda^1\}$ satisfies (CK2). Now take $f \in \Lambda^1$. If $s(f) \neq v$ we have
$$T_f^*T_f = s_f^*s_f = s_{s(f)} = T_{s(f)}.$$ 
If $s(f) = v$ we have $\{ g: par(g) = f \} = \{ f^1, f^2\}$, so the fact that $v^1 = s(f^1) \not= s(f^2) = v^2$ implies that 
\begin{equation*}
\begin{split}
T_f^*T_f 
& (s_{f^1} + s_{f^2})^* (s_{f^1} + s_{f^2})= s_{f^1}^* s_{f^1} + s_{f^2}^* s_{f^2} = s_{v^1} + s_{v^2} = T_v.
\end{split}
\end{equation*}
Thus $\{T_\lambda : \lambda \in \Lambda^0 \cup \Lambda^1\}$ satisfies (CK3). Finally fix a generator $e_{i} \in \N^{k}$ and fix $w \in \Lambda^0$.  
We first observe that if two distinct edges in $\Lambda_{I}^1$ have the same parent, they must have different sources (namely $v_1$ and $v_2)$ and orthogonal range projections, and therefore, for $\lambda \in \Lambda^1$,
\begin{equation*}
\begin{split}
& \left( \sum_{par(e) = \lambda} s_{e} \right) \left( \sum_{par(e) = \lambda} s_{e}^* \right) = \sum_{par(e) = \lambda} s_{e}s_{e}^*.
\end{split}
\end{equation*}
It follows that 
\begin{equation*}
\begin{split}
\sum_{\substack{d(e) = e_i\\ r(e) = w}} T_e T_e^* 
&= \sum_{\substack{d(e) = e_i \\ r(e) = w}} \left( \sum_{par(f) = e} s_f \right) \left(\sum_{par(f) = e} s_f^* \right) = \sum_{\substack{d(e) = n_i \\ r(e) = w}} \sum_{par(f) = e} s_f s_f^* \\
&= \sum_{\substack{d(f) = e_i \\ r(par(f)) = w}} s_f s_f^* = \sum_{par(x) = w} s_x = T_w.
\end{split}
\end{equation*}
\\
Therefore $\{T_\lambda : \lambda \in \Lambda^0 \cup \Lambda^1\}$ satisfies (CK4), and hence is a Cuntz--Krieger $\Lambda$ family in $C^*(\Lambda_I)$. Thus, by the universal property of $C^*(\Lambda)$, there exists a $*$-homomorphism $\pi: C^*(\Lambda) \to C^*(\Lambda_I)$ such that $\pi(t_\lambda) = T_{\lambda}$, where $\{t_\lambda : \lambda \in \Lambda^0 \cup \Lambda^1\}$ are the canonical generators of $C^*(\Lambda)$.  We will show that $\pi$ is an isomorphism.

Fix $w \in \Lambda_I^0$ and note that if $par(w) \neq v$ then $s_{w} = T_{w} \in \pi(C^*(\Lambda))$. Conversely if $par(w) = v$ then $w = v^j$ for some $j \in \{1,2\}$. Thus, for a fixed generator $e_i \in \N^k$ we have
\begin{equation*}
\begin{split}
\sum_{\substack{r(e) = v \\ d(e) = e_i \\ e \in \E_j}} T_e T_e^* 
= \sum_{\substack{r(e) = v \\ d(e) = e_i \\ e \in \E_j}} \sum_{par(e') = e} s_{e'} s_{e'}^* 
= \sum_{\substack{r(e') = v^j \\ d(e') = e_i}} s_{e'} s_{e'}^*
= s_{v^j} \in \pi(C^*(\Lambda)).
\end{split}
\end{equation*}
Thus, all of the vertex projections of $C^*(\Lambda_I)$ are in $\text{Im}(\pi)$, and therefore $\pi(C^*(\Lambda))$ contains all of the generators of $C^*(\Lambda_{I})$. Hence $\pi$ is surjective. 

Consider the canonical gauge actions $\alpha$ of $\mathbb{T}^k$ on $C^*(\Lambda_I)$ and $\beta$ of $\mathbb{T}^k$ on $C^*(\Lambda)$.  Observe that for all $z \in \mathbb{T}^k,$ the fact that $par$ is degree-preserving implies that
\[ \alpha_{z}(T_{\lambda}) = \sum_{par(\mu) = \lambda} \alpha_z(s_\mu) = z^{d(\lambda)} \sum_{par(\mu) = \lambda } s_\mu = z^{d(\lambda)} T_\lambda.\]
Therefore, 
\begin{equation*}
\pi(\beta_{z}(t_{\lambda})) = \pi(z^{d(\lambda)}t_{\lambda}) = z^{d(\lambda)}T_{\lambda} = \alpha_{z}(T_{\lambda}) = \alpha_{z}(\pi(t_{\lambda})),   
\end{equation*}
so $\pi$ intertwines the canonical gauge actions. The gauge invariant uniqueness theorem now implies that $\pi$ is injective, and so $C^*(\Lambda) \cong C^*(\Lambda_I)$ as claimed.
\end{proof}
\section{Delay}
\label{sec:delay}

Our goal in this section is to generalize to $k$-graphs the operation of delaying a graph at an edge -- that is, breaking an edge in two by adding a vertex in the ``middle'' of the edge. The importance of this construction can be traced back to Parry and Sullivan's analysis \cite{parry-sullivan} of flow equivalence for shifts of finite type; Drinen  realized \cite{drinen-thesis} that in the setting of directed graphs, these edge delays correspond to the expansion matrices used by Parry and Sullivan to complete the charaterization of flow equivalence for shifts of finite type.   Bates and Pask later generalized the ``delay'' operation in \cite{bates-pask} and showed that the $C^*$-algebra of a delayed graph is Morita equivalent to the $C^*$-algebra of the original graph. 

In the setting of higher-rank graphs, the ``delay'' operation becomes more intricate.  So that the resulting object satisfies the factorization rule, after delaying one edge and adding a new vertex, new edges of other colors (incident with the new vertex) must be added. This procedure is described in Definition \ref{def:delay} below, and Theorem \ref{thm:DelayKG} establishes that the resulting object $\Lambda_D$ is indeed a $k$-graph. Theorem \ref{thm:DelayME} then proves that $C^*(\Lambda_D)$  is  Morita equivalent to the $C^*$-algebra of the original $k$-graph $\Lambda$.

\begin{definition}
\label{def:delay}
Let $(\Lambda, d)$ be a $k$-graph and $G = (\Lambda^0, \Lambda^1, r, s)$ its underlying directed graph.  Fix $f \in \Lambda^1$; without loss of generality, assume $d(f) = e_1$.   We first recursively define the set $\E^{e_1}$ of all possible elements of $\Lambda^{e_1}$ which will be affected by delaying $f$, in that elements of $\E^{e_1}$ are opposite to $f$ in some commuting square in $\Lambda$.  Namely, we set 
\begin{align*}
A_1 &= \{ f \} \cup  \{ g \in \Lambda^{e_1} : ag \sim fb \text{ or } ga \sim bf  \text{ where } a,b \in \Lambda^{e_i} \text{ for }  2 \leq i \leq k   \}, \\
A_m &= \{ e \in \Lambda^{e_1} : ag \sim eb \text{ or } ga \sim be
 \text{ where } a,b \in \Lambda^{e_i} \text{ for }  2 \leq i \leq k , ~ g \in A_{m-1} \}, \\
\E^{e_1} &= \bigcup_{j = 1}^{\infty} A_j  \subseteq \Lambda^{e_1}. 
\end{align*}

In the pictures below, the dashed edges would all lie in $\E^{e_1}$.
\[
\begin{tikzcd} [column sep=.7cm, row sep=.7cm]
	& |[alias = A]| \bullet \ar[r] \ar[d, dashed, black,  "f" '] \ar[d,phantom,shift left=3ex, "\alpha"'] & |[alias = C]| \bullet \ar[d, dashed]\ar[d,phantom,shift left=3ex, "\beta"'] \ar[r] & |[alias = E]| \bullet \ar[d, dashed]\ar[d,phantom,shift left=3ex, "\gamma"'] \ar[r] & |[alias = H]| \bullet \ar[d, dashed] \\
	& |[alias = B]| \bullet \ar[r,  near start,] & |[alias = D]| \bullet \ar[r, near start] & |[alias = F]| \bullet \ar[r,  near start] & |[alias = G]| \bullet
\end{tikzcd}
~~~~\text{ or }
\begin{tikzcd} [column sep=.7cm, row sep=.7cm]
	& |[alias = A]| \bullet \ar[r] \ar[d, dashed, near start]\ar[d,phantom,shift left=3ex, "\xi"'] & |[alias = C]| \bullet \ar[d, dashed] \ar[d,phantom,shift left=3ex, "\eta"']\ar[r] & |[alias = E]| \bullet \ar[d, dashed]\ar[d,phantom,shift left=3ex, "\zeta"'] \ar[r] & |[alias = H]| \bullet \ar[d, black, dashed, "f"] \\
	& |[alias = B]| \bullet \ar[r,  near start,] & |[alias = D]| \bullet \ar[r, near start] & |[alias = F]| \bullet \ar[r,  near start] & |[alias = G]| \bullet
\end{tikzcd}
\]

\noindent
Using $\E^{e_1}$ we identify those commuting squares of degree $(e_1 + e_i)$, $i \neq 1$ in $\Lambda$ which contain an edge from $\E^{e_1}$.  These squares will form the set $\E^{e_i}:$ 
\[
\E^{e_i} = \{ [ga] \in \Lambda : g \in \E^{e_1}, a \in \Lambda^{e_i} \} . 
\]
 In the pictures above, if the solid black edges have degree $e_i$, we have $\alpha,\beta,\gamma,\xi, \eta, \zeta \in \mathcal E^{e_i}.$   By delaying $f$, these squares   will be turned into rectangles.
\noindent

To be precise, in the delayed graph, we will ``delay'' every edge in $\E^{e_1},$ replacing it with two edges:
\[
\begin{tikzcd} [column sep=.7cm, row sep=.7cm]
	& |[alias = A]| \bullet \ar[r] \ar[d, dashed, near start, "f^1" ']  & |[alias = C]| \bullet \ar[d, dashed]\ \ar[r] & |[alias = E]| \bullet \ar[d, dashed] \ar[r] & |[alias = H]| \bullet \ar[d, dashed] \\
	& |[alias = M]| \bullet  \ar[d, dashed,near start, "f^2" ']  & |[alias = K]| \bullet \ar[d, dashed] & |[alias = L]| \bullet \ar[d, dashed]& |[alias = H]| \bullet \ar[d, dashed] \\
	& |[alias = B]| \bullet \ar[r,  near start,] & |[alias = D]| \bullet \ar[r, near start] & |[alias = F]| \bullet \ar[r,  near start] & |[alias = G]| \bullet
\end{tikzcd}
~~~~\text{ or }
\begin{tikzcd} [column sep=.7cm, row sep=.7cm]
	& |[alias = A]| \bullet \ar[r] \ar[d, dashed, near start] & |[alias = C]| \bullet \ar[d, dashed] \ar[r] & |[alias = E]| \bullet \ar[d, dashed] \ar[r] & |[alias = H]| \bullet \ar[d, dashed, "f^1"] \\
	& |[alias = M]| \bullet  \ar[d, dashed]  & |[alias = K]| \bullet \ar[d, dashed] & |[alias = L]| \bullet \ar[d, dashed]& |[alias = H]| \bullet \ar[d, dashed, "f^2"] \\
	& |[alias = B]| \bullet \ar[r,  near start,] & |[alias = D]| \bullet \ar[r, near start] & |[alias = F]| \bullet \ar[r,  near start] & |[alias = G]| \bullet
\end{tikzcd}
\]
\[
\E_D^{e_1} = \{ g^1, g^2 : g \in \E^{e_1} \} 
\]

\noindent
and add an edge for every square that has been turned into a rectangle.
\[
\begin{tikzcd} [column sep=.7cm, row sep=.7cm]
	& |[alias = A]| \bullet \ar[r] \ar[d, dashed, near start, "f^1" ']  & |[alias = C]| \bullet \ar[d, dashed]\ \ar[r] & |[alias = E]| \bullet \ar[d, dashed] \ar[r] & |[alias = H]| \bullet \ar[d, dashed] \\
	& |[alias = M]| \bullet  \ar[d, dashed,near start, "f^2" ']  \ar[r,"e_\alpha"] & |[alias = K]| \bullet \ar[d, dashed] \ar[r,"e_\beta"] & |[alias = L]| \bullet \ar[d, dashed] \ar[r,"e_\gamma"] & |[alias = H]| \bullet \ar[d, dashed] \\
	& |[alias = B]| \bullet \ar[r,  near start,] & |[alias = D]| \bullet \ar[r, near start] & |[alias = F]| \bullet \ar[r,  near start] & |[alias = G]| \bullet
\end{tikzcd}
~~~~\text{ or }
\begin{tikzcd} [column sep=.7cm, row sep=.7cm]
	& |[alias = A]| \bullet \ar[r] \ar[d, dashed, near start] & |[alias = C]| \bullet \ar[d, dashed] \ar[r] & |[alias = E]| \bullet \ar[d, dashed] \ar[r] & |[alias = H]| \bullet \ar[d, dashed, "f^1"] \\
	& |[alias = M]| \bullet  \ar[d, dashed]  \ar[r,"e_\xi"] & |[alias = K]| \bullet \ar[d, dashed]  \ar[r,"e_\eta"]& |[alias = L]| \bullet \ar[d, dashed] \ar[r,"e_\zeta"] & |[alias = H]| \bullet \ar[d, dashed, "f_2"] \\
	& |[alias = B]| \bullet \ar[r,  near start,] & |[alias = D]| \bullet \ar[r, near start] & |[alias = F]| \bullet \ar[r,  near start] & |[alias = G]| \bullet
\end{tikzcd}
\]
\[
\E_D^{e_i} = \{ e_{\alpha} : \alpha \in \E^{e_i} \}.
\]

Then define the $k$-colored graph $G_D = (\Lambda_D^0, \Lambda_D^1, r_D, s_D)$ by
\begin{equation*}
\begin{split}
\Lambda_D^0 &= \Lambda^0 \cup \{ v_g \}_{g \in \E^{e_1}},  
\quad \Lambda_D^{e_1} = (\Lambda^{e_1} \setminus \E^{e_1}) \cup \E_D^{e_1} \text{, with} \\ 
& \hspace{10mm} s_D(g^1) = s(g), s_D(g^2) = v_g, 
 \hspace{3mm} r_D(g^1) = v_g, r_D(g^2) = r(g) ; \\ 
\Lambda_D^{e_i}& = \Lambda^{e_i} \cup \E_D^{e_i} \text{, with} \\ 
& \hspace{10mm} s_D(e_{\alpha}) = v_{g} \text{ such that } bg \text{ represents } \alpha \text{ and } d(g) = e_1, \\
& \hspace{10mm} r_D(e_{\alpha}) = v_{h} \text{ such that } ha \text{ represents } \alpha \text{ and } d(h) = e_1.
\end{split}
\end{equation*}

In words, to construct $G_D$ from $G$, we add one vertex per delayed edge; each delayed edge becomes two edges in $G_D$; and we add one edge  for each square that was  stretched into a rectangle by delaying the edges in $\E^1$.  If $\alpha \in \E^{e_i}$ we set $d(e_{\alpha}) = e_i$, and all other edges inherit their degree from $\Lambda$.

Let $\iota : G_D \to G$ be the partially defined inclusion map with domain $(\Lambda_D^0 \cup \Lambda_D^1) \setminus (\{\bigcup\limits_{i=1}^k \E_D^{e_i}\} \cup \{v_e : e \in \E^{e_1}\})$. 
Then, for edges $ g\in \Lambda^1_D \backslash \bigcup_{i=1}^k \E_D^{e_i}$, we can define
\begin{equation*}
s_D(g) = s(\iota(g)), ~r_D(g) = r(\iota(g)), ~d_D(g) = d(\iota(g)) .
\end{equation*}

Now let $G^*_D$ be the path category for $G_D$ and define the equivalence relation $\sim_D$ on bi-color paths  $\mu = \mu_2 \mu_1  \in G_D^2$ according to the following rules. 
\\

Case 1: Assume $\mu_1, \mu_2 \notin \bigcup_{i=1}^k \E_D^{e_i}$. Then we set $[\mu]_D = \iota^{-1}([\iota(\mu)])$.
\\

Case 2: Suppose  $\mu_j$ lies in $ \E^{e_1}_D$, so that $\mu_j \in \{ g^1, g^2\}$ for some edge $g \in \E^{e_1}$.  If $j=1$ and $\mu_1 = g^2$, then $r(\mu_1) = s(\mu_2) = \iota^{-1}( r(g)) \in \iota^{-1}(\Lambda^0)$, and the edges in $G_D$ with source in $\iota^{-1}(\Lambda^0)$ and degree $e_i$ for $i \not= 1$  are in $\iota^{-1}(\Lambda^1)$.  Therefore $\mu_2 \in \iota^{-1}(\Lambda^{e_i})$, and $\iota(\mu_2) g$ is a bi-color path in $G$, so $\iota(\mu_2) g \sim h a$ for edges $h \in \E^{e_1}, a \in \Lambda^{e_i}$.  There is then an edge $e_{[\mu_2 g]} \in \Lambda^{e_i}_D$ with source $s(\mu_1) = v_g$ and range $v_h = s(h^2)$;  we define $\mu_2 \mu_1 =  \mu_2 g^2  \sim_D h^2 e_{[\mu_2 g]}$.
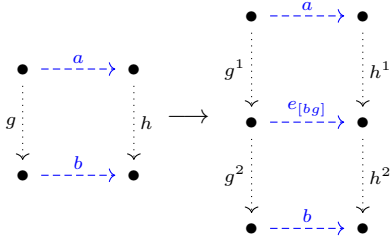
\begin{figure}[H]
\begin{tikzcd}[column sep=1cm,row sep=1cm]
\bullet \ar[r, dashed, blue, "a"] \ar[d, dotted, "g"'] & \bullet \ar[d, dotted, "h"]\\
\bullet \ar[r, dashed, blue, "b"] & \bullet
\end{tikzcd}
$\longrightarrow$
\begin{tikzcd}[column sep=1cm,row sep=1cm]
\bullet \ar[r, dashed, blue, "a"] \ar[d, dotted, "g^1"'] & \bullet \ar[d, dotted, "h^1"]\\
\bullet \ar[r, dashed, blue, "e_{[bg]}"] \ar[d, dotted, "g^2"'] & \bullet \ar[d, dotted, "h^2"]\\
\bullet \ar[r, dashed, blue, "b"] & \bullet
\end{tikzcd}
    \caption{A commuting square in $G$ and its ``children" in $G_D$, when $h,g \in \E^{e_1}$.}
    \label{fig:Case2DelayDef}
\end{figure}
If $j =1$ and $\mu_1 = g^1$, the only edges in $G_D$ with source $r(g^1) = v_g $ and degree $e_i$ for $i \not=1$ are of the form $e_{[bg ]} = e_{[ha]}$ for some commuting square $bg \sim ha $ in $ \Lambda$.  In this case, we will have $h \in \E^{e_1}$, and $r(h^1) = v_h = r(e_{[bg]})$, so we set $e_{[bg]} g^1 \sim_D  h^1 a$.

A similar argument shows that if $j=2$, the path $\mu_2 \mu_1$ will be of the form $h^1 a$ or $h^2 e_{[ha]}$, whose factorizations we have already described.
\\

Case 3: Assume $\mu$ is of the form $e_{\beta} e_{\alpha}$ for $\alpha \in \E_D^{e_i}$, and $\beta \in \E_D^{e_j}$ with $i \neq j$. Now $s_D(e_{\beta}) = r_D(e_{\alpha}) = v_g$ for some $g \in \E^{e_1}$, and consequently $\alpha, \beta \in \Lambda$ are linked as shown on the left of Figure \ref{fig:Case3DelayDef}. Since $\Lambda$ is a $k$-graph, the 3-color path outlining $\beta \alpha$ generates a 3-cube in $\Lambda$, which is depicted on the right of Figure \ref{fig:Case3DelayDef}. 
\begin{figure}[H]
\includegraphics[width=3in]{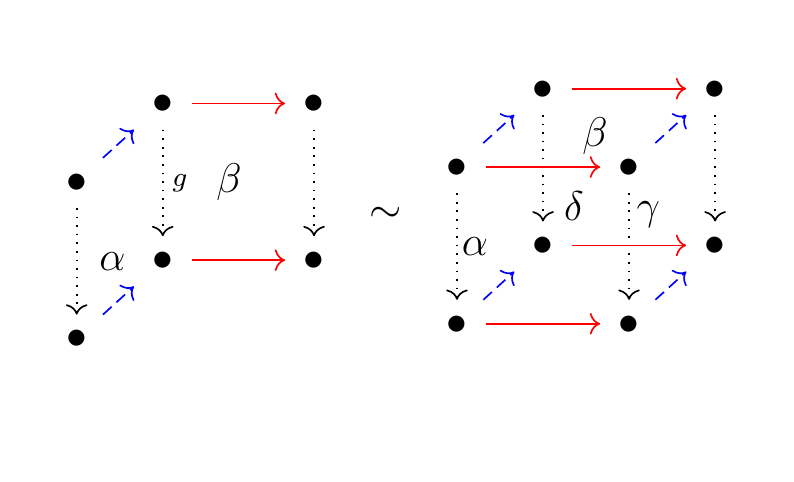}
    \caption{The commuting squares of edges from $\bigcup_{i=2}^k \E_D^{e_i}$.}
    \label{fig:Case3DelayDef}
\end{figure}
Let $\delta$ and $\gamma$ denote the faces of this cube which lie, respectively, opposite $\beta$ and $\alpha$. Since $g \in \E^{e_1}$,  all of the vertical edges of this cube are in $\E^{e_1},$ and so $\delta \in \E^{e_j}, \gamma \in \E^{e_i}.$ Moreover, the path $e_{\gamma} e_{\delta}$ is composable in $\Lambda_D$, and has the same source and range as $e_{\beta} e_{\alpha}$.  Set $e_{\beta} e_{\alpha} \sim_D e_{\gamma} e_{\delta}$.

Observe that there are no two-color paths in $G_D^2$  of the form $g e_{\alpha}$ or $e_{\alpha} g$ for $g \in \iota^{-1}(\Lambda^1)$ and $\alpha \in \E^{e_i}$, since $ r_D(g) , s_D(g) \in \iota^{-1}(\Lambda^0)$ but $r_D(e_{\alpha}), s_D(e_{\alpha}) \in \{ v_e: e \in \E^{e_1}\}.$
\\

Extend $\sim_D$ to be an equivalence relation on $G^*_D$ which satisfies (KG0) and (KG1); observe that $\sim_D$ satisfies (KG2) by construction. Define $\Lambda_D = G_D^* / \sim_D$. We call $\Lambda_D$ the graph of $\Lambda$ delayed at the edge $e$.
\end{definition}

\begin{theorem}
If $\Lambda$ is a row-finite source-free $k$-graph, then $\Lambda_D$ is also a source-free $k$-graph.
\label{thm:DelayKG}
\end{theorem}

\begin{proof}
Let $(\Lambda, d)$ be a \textit{k}-graph, and let $(\Lambda_D, d_D)$ be the graph of $\Lambda$ delayed at the edge $e \in \Lambda^{e_1}$. Since $\sim_D$ satisfies (KG0), (KG1), and (KG2) by construction,  it suffices to show that $\sim_D$ satisfies (KG3). Let $\mu = \mu_3 \mu_2 \mu_1 \in G_D^3$ be a tri-colored path. Consider the following cases.
\\

Case 1: Assume $\mu_j \notin \bigcup\limits_{i=1}^k \E_D^{e_i}$ for all $j \in \{1,2,3\}$. Then $\iota(\mu)$ is a 3-colored path in $\Lambda$.  Since we defined $[\mu]_D = \iota^{-1}([\iota(\mu)])$, the fact that $\Lambda$ satisfies (KG3) --  hence, that $\iota(\mu)$ uniquely determines a 3-cube in $\Lambda$ -- implies that $\mu$ also gives rise to a well-defined 3-cube in $\Lambda_D$. 
\\

Case 2: Assume $\mu_j \in \E_D^{e_1}$ for one $j \in \{ 1,2,3\}$.  
Then $\mu$ has one of the forms of tri-colored paths on the right-hand side of Figure \ref{fig:Case2DelayKG}. This follows from Definition \ref{def:delay}, specifically the restrictions for when an edge in $\iota^{-1}(\Lambda^1)$ can precede or follow an edge in $\E_D^{e_1}$, and when an edge in $\E_D^{e_i}$ can precede or follow an edge in $\E_D^{e_1}$. 

For example, suppose $\mu_2 = g^2$ for some $g \in \mathcal E^{e_1}$.  Then $s_D(\mu_2) = r_D(\mu_1) = v_g$, so $\mu_1$ must be of the form $e_{\gamma}$ for some $\gamma \in \E^{e_i}$ with $\gamma = [g q]$.  Then $r_D(\mu_2) = s_D(\mu_3) = \iota^{-1}(r(g))$, and since $d(\mu_3) \not\in \{e_1, e_i\}$ we must have $\mu_3 
\in \iota^{-1}(\Lambda^{e_j})$ for some $j \not= i$.  But then, $\iota^{-1}(\mu_3) g q \in G^3$ is a 3-color path, which defines a unique 3-cube in $\Lambda$ (as depicted on the left of Figure \ref{fig:Case2DelayKG}).  It is now  straightforward to check that the two routes for factoring $\mu$ in $\Lambda_D$ (as in (KG3)) arise as ``children'' of this cube in $G_D$, so the fact that $\Lambda$ is a $k$-graph implies that the two routes for factoring $\mu$ in $\Lambda_D$ lead to the same result.  A similar analysis of the other possibilities for having one edge $\mu_j \in \E^{e_1}_D$  reveals that whenever this occurs, the factorization of $\mu$ in $\Lambda_D$ satisfies (KG3).

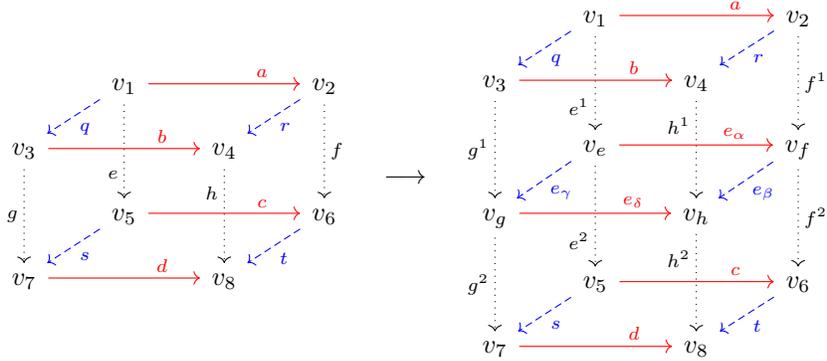
\begin{figure}[H]
\begin{tikzcd}[scale=0.8,column sep=.7cm,row sep=.4cm]
	 & v_1   \ar[rr, near end, red, "a"]\ar[dd, black, dotted, near end, "e"'] \ar[dl, blue, dashed, "q"] & & v_2  \ar[dl, blue, dashed, "r"] \ar[dd, black, dotted, "f"]\\
	v_3 \ar[rr, red, near end, "b"]\ar[dd, black, dotted, "g"']  & & v_4  \ar[dd, black, near start, dotted, "h"'] & \\
	& v_5 \ar[dl, blue, dashed, "s"] \ar[rr, red, near end, "c"]  & & v_6  \ar[dl, blue, dashed, "t"] \\
	v_7 \ar[rr, near end, red, "d"] & & v_8 
\end{tikzcd}
$\quad \longrightarrow \quad$
\begin{tikzcd} [column sep=.7cm,row sep=.4cm]
& v_1 \ar[rr, near end, red, "a"] \ar[dd, black, dotted, near end, "e^1"'] \ar[dl, blue, dashed, "q"] & & v_2 \ar[dl, blue, dashed, "r"] \ar[dd, black, dotted, "f^1"]\\
v_3 \ar[dd, black, dotted, "g^1"'] \ar[rr, red, near end, "b"] & & v_4 \ar[dd, black, dotted, near start, "h^1"'] & \\
& v_e \ar[rr, near end, red, "e_\alpha"] \ar[dd, black, dotted, near end, "e^2"'] \ar[dl, blue, dashed, "e_\gamma"] & & v_f \ar[dl, blue, dashed, "e_\beta"] \ar[dd, black, dotted, "f^2"]\\
v_g \ar[dd, black, dotted, "g^2"'] \ar[rr, red, near end, "e_\delta"] & & v_h \ar[dd, black, dotted, near start, "h^2"'] & \\
& v_5 \ar[dl, blue, dashed, "s"] \ar[rr, red, near end, "c"] & & v_6 \ar[dl, blue, dashed, "t"] \\
v_7 \ar[rr, near end, red, "d"] & & v_8 &
\end{tikzcd}
    \caption{A commuting cube in $G$ and its ``children" in $G_D$, when $e,f,g,h \in \E_D^{e_1}$}
    \label{fig:Case2DelayKG}
\end{figure}

We now observe that if a tri-colored path without an edge from $\E^{e_1}_D$ contains  an edge from $\bigcup_{i=2}^k \E^{e_i}_D,$ it must consist entirely of edges in $\bigcup_{i=2}^k \E^{e_i}_D$.  To see this, suppose that a tri-colored path contains $e_\gamma$ for a commuting square $\gamma \in \Lambda$, but that $\mu$ contains no edges in $\E^{e_1}_D$.   Since $s(\gamma), r(\gamma) \in \Lambda^0_D \backslash \iota^{-1}(\Lambda^0)$, the edge(s) preceding and following $e_\gamma$ must be of the form $e_\alpha$ for some $\alpha \in \E^{e_i}$. Repeating the argument for $e_\alpha$ if necessary shows that $\mu$ consists entirely of edges in $\bigcup_{i=2}^k \E^{e_i}_D$.

Thus, the only remaining case is

Case 3: Assume $\mu_j \in \bigcup\limits_{i=2}^k \E_D^{e_i}$ for all $j \in \{ 1,2,3\}$, and without loss of generality assume $\mu = e_\alpha e_\beta e_\gamma$ is a blue-red-green path. 
Because of the definition of $s_D, r_D$ for edges of the form $e_\lambda$ in $G_D^1$, $\mu \in G_D^*$ arises from a sequence of commuting squares $\alpha, \beta, \gamma$ in $\Lambda$ which share edges in $\E^{e_1}$.  Figure \ref{fig:Case3DelayKG1} below depicts (from left to right) the squares $\gamma, \beta, \alpha \in \Lambda$.  The color in each square $\lambda$ reflects the color of its horizontal edges, as these determine the degree of the edge $e_\lambda \in G_D^1$.

\begin{figure}[H]
\includegraphics[width=2in]{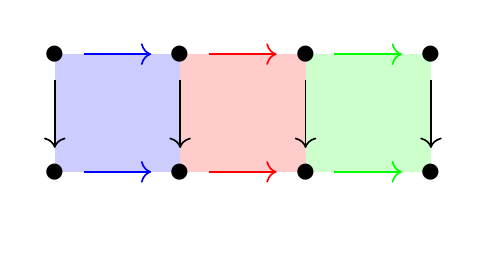}
    \caption{Factorization squares in $\Lambda$ that will be delayed to produce $\mu$}
    \label{fig:Case3DelayKG1}
\end{figure}

Because $\Lambda$ is a $k$-graph, the rectangle in Figure \ref{fig:Case3DelayKG1}, which is reproduced in the top line of Figure 
8,
determines a unique 4-dimensional cube in $\Lambda$.  Thus, as we follow Route 1 of (KG3) and the instructions given in Case 3 of Definition \ref{def:delay} to factor 
\[e_\alpha e_\beta e_\gamma = e_\alpha e_\eta e_\kappa = e_\delta e_\epsilon e_\kappa = e_\delta e_\phi e_\lambda,\]
the squares $\delta \phi \lambda$ -- and indeed all of the intermediate squares -- must lie on the 4-dimensional cube determined by $\alpha \beta \gamma$.  To be precise, $\delta \phi \lambda$ is  the collection of green-red-blue squares on the left of the bottom row of Figure 
8.  Similarly, when we factor $e_\alpha e_\beta e_\gamma$ via Route 2 of (KG3), we obtain the green-red-blue squares on the right of the bottom row of Figure 
8.  Because these squares lie on the same 4-cube as $\delta \phi \lambda$, and in the same position (compare the position of $\nu_0$ on both), they must equal $\delta \phi \lambda$. It follows that applying either Route 1 or Route 2 to $e_\alpha e_\beta e_\gamma$ gives us  the same 3-colored path in $G_D$.

\begin{figure}[H]
\begin{tikzcd} [column sep=.7cm, row sep=.7cm]
	& |[alias = A]| \bullet \ar[r, blue] \ar[d, black, near start, "\nu_0" '] & |[alias = C]| \bullet \ar[d, black] \ar[r, red] & |[alias = E]| \bullet \ar[d, black] \ar[r, green] & |[alias = H]| \bullet \ar[d, black] \\
	& |[alias = B]| \bullet \ar[r, blue, near start, "\nu_1" '] & |[alias = D]| \bullet \ar[r, red, near start, "\nu_2" '] & |[alias = F]| \bullet \ar[r, green, near start, "\nu_3" '] & |[alias = G]| \bullet
\end{tikzcd}

$\swarrow \hspace{3cm}\searrow$

\begin{tikzcd}[column sep=.3cm,row sep=.3cm]
& \bullet \ar[dl, blue] \ar[rr, red] \ar[dd, black, near start, "\nu_0"] & & \bullet \ar[dl, blue] \ar[dd, black] & \\
\bullet \ar[rr, red] \ar[dd, black] & & \bullet \ar[rr, green] \ar[dd, black] & & \bullet \ar[dd, black] \\
& \bullet \ar[rr, red] \ar[dl, blue, near start, "\nu_1", '] & & \bullet \ar[dl, blue] &  \\
\bullet \ar[rr, red, "\nu_2"] & & \bullet \ar[rr, green, "\nu_3"] & & \bullet  
\end{tikzcd}
\hspace{3cm}
\begin{tikzcd}[column sep=.3cm,row sep=.3cm]
& & & \bullet \ar[rr, red] \ar[dd, black]  &  & \bullet \ar[dd, black] \\
\bullet \ar[rr, blue] \ar[dd, black, , "\nu_0"] & & \bullet \ar[ur, green] \ar[rr, red] \ar[dd, black] & & \bullet \ar[ur, green] \ar[dd, black]  &  \\
& & & \bullet \ar[rr, red] & & \bullet \\
\bullet \ar[rr, blue, "\nu_1"] & & \bullet \ar[rr, red, "\nu_2"] \ar[ur, green] & & \bullet \ar[ur, green, "\nu_3", '] &  
\end{tikzcd}

$\downarrow \hspace{7cm} \downarrow$

\begin{tikzcd}[column sep=.3cm,row sep=.3cm]
& \bullet  \ar[dl, blue] \ar[rr, red] \ar[dd, black, near start, "\nu_0"] & & \bullet \ar[dl, blue] \ar[rr, green] \ar[dd, black] & & \bullet \ar[dd, black] \ar[dl, blue] \\    
\bullet \ar[rr, red] \ar[dd, black] & & \bullet \ar[rr, green] \ar[dd, black] & & \bullet \ar[dd, black]  &\\
& \bullet \ar[rr, red] \ar[dl, blue, near start, "\nu_1", '] & & \bullet \ar[dl, blue] \ar[rr, green] & & \ar[dl, blue] \\
\bullet \ar[rr, red, "\nu_2"] & & \bullet \ar[rr, green, "\nu_3"] & & \bullet & 
\end{tikzcd}
\hspace{3cm}
\begin{tikzcd}[column sep=.3cm,row sep=.3cm]
& \bullet \ar[rr, blue] \ar[dd, black] & & \bullet \ar[rr, red] \ar[dd, black]  & & \bullet \ar[dd, black] \\    
\bullet \ar[rr, blue] \ar[dd, black, "\nu_0"] \ar[ur,green]  & & \bullet \ar[rr, red] \ar[dd, black] \ar[ur,green] & & \bullet \ar[dd, black] \ar[ur,green]  &\\
& \bullet \ar[rr, blue] & & \bullet \ar[rr, red] & & \bullet \\
\bullet \ar[rr, blue, "\nu_1"] \ar[ur,green] & & \bullet \ar[rr, red, "\nu_2"] \ar[ur,green] & & \bullet \ar[ur,green, "\nu_3", '] & 
\end{tikzcd}

$\downarrow \hspace{7cm} \downarrow$

\begin{tikzcd}[column sep=.25cm,row sep=.25cm,
execute at end picture={
  \begin{pgfonlayer}{background}
  \foreach \Nombre in  {A,B,C,D,E,F,G,H}
    {\coordinate (\Nombre) at (\Nombre.center);}
  \fill[green!20] 
    (A) -- (B) -- (D) -- (C) -- cycle;
      \fill[red!20] 
    (D) -- (C) -- (E) -- (F) -- cycle;
    \fill[blue!20] 
    (E) -- (F) -- (G) -- (H) -- cycle;
    \end{pgfonlayer}
}
]
& & & & & |[alias = C]| \bullet \ar[drr, red] \ar[ddd, black] & & & \\
& |[alias = A]| \bullet \ar[urrrr, green]  \ar[dl, blue] \ar[rrr, red] \ar[ddd, black, "\nu_0"'] & & & \bullet \ar[dl, blue] \ar[rrr, green] \ar[ddd, black] & & & |[alias = E]| \bullet \ar[ddd, black] \ar[dl, blue] \\    
 \bullet \ar[rrr, red] \ar[ddd, black] & & & \bullet \ar[rrr, green] \ar[ddd, black]  & & & |[alias = H]| \bullet \ar[ddd, black]   &\\
& & & & &  |[alias = D]| \bullet \ar[drr , red] &  & &  \\ 
& |[alias = B]| \bullet \ar[urrrr, green] \ar[rrr, red] \ar[dl, blue, near start , "\nu_1"' ] & & & \bullet \ar[dl, blue] \ar[rrr, green] & & & |[alias = F]| \bullet \ar[dl, blue] \\
\bullet \ar[rrr, red, "\nu_2"']  & & & \bullet \ar[rrr, green, "\nu_3"'] & & & |[alias = G]| \bullet & 
\end{tikzcd}
$~~~=~~~$
\begin{tikzcd}[column sep=.2cm,row sep=.2cm,
execute at end picture={
  \begin{pgfonlayer}{background}
  \foreach \Nombre in  {A,B,C,D,E,F,G,H}
    {\coordinate (\Nombre) at (\Nombre.center);}
  \fill[green!20] 
    (A) -- (B) -- (D) -- (C) -- cycle;
      \fill[red!20] 
    (D) -- (C) -- (E) -- (F) -- cycle;
    \fill[blue!20] 
    (E) -- (F) -- (G) -- (H) -- cycle;
    \end{pgfonlayer}
}
]
& & & & & |[alias = E]| \bullet \ar[drr, blue] \ar[ddd, black] & & & \\
& |[alias = C]| \bullet \ar[urrrr, red]  \ar[rrr, blue] \ar[ddd, black] & & &  \bullet \ar[rrr, red] \ar[ddd, black] & & & |[alias = H]| \bullet \ar[ddd, black]  \\    
|[alias = A]| \bullet \ar[ur, green] \ar[rrr, blue] \ar[ddd, black, "\nu_0"'] & & &  \bullet \ar[rrr, red] \ar[ddd, black] \ar[ur, green] & & & \bullet \ar[ddd, black] \ar[ur, green]  &\\
& & & & & |[alias = F]| \bullet \ar[drr , blue] &  & &  \\ 
&  |[alias = D]| \bullet \ar[urrrr, red] \ar[rrr, blue] & & & \bullet \ar[rrr, red] & & & |[alias = G]| \bullet \\
|[alias = B]| \bullet \ar[rrr, blue, "\nu_1"'] \ar[ur, green]  & & & \bullet \ar[rrr, red, "\nu_2"'] \ar[ur, green]  & & &  \bullet \ar[ur, green, "\nu_3"'] & 
\end{tikzcd}

\label{fig:Case3DelayKG2}
\caption{Associativity  in $\Lambda_D$ via factorization squares in $\Lambda$}
\end{figure}
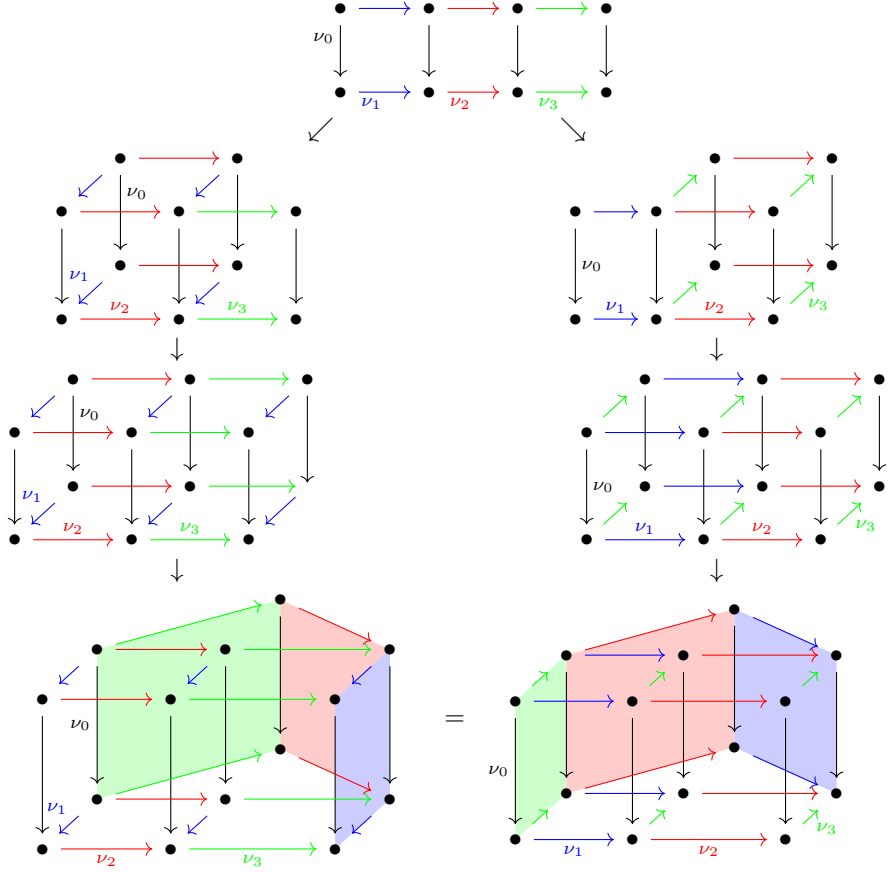
Having confirmed that  the factorization of an arbitrary  tri-colored path in $G_D$ satisfies (KG3), we conclude that $\Lambda_D$ is a $k$-graph.

It remains to check that $\Lambda_D$ is source-free and row-finite whenever $\Lambda$ is.  In constructing $\Lambda_D$, all newly-created vertices $v_g$ have an edge $g^2$ of color 1 emanating from them.  Moreover,  since $r(g)$ is not a source in $\Lambda$, there is an edge $b_i \in \Lambda^{e_i} r(g)$ for each $i \geq 2$.  Then, $[b_i g] \in \E^{e_i}$ and hence $e_{[b_i g]}\in \Lambda^{e_i}_D v_g$ for all $i \geq 2$.  In other words, all of the new vertices $v_g$ emit at least one edge of each color.

Similarly, every vertex $v \in \iota^{-1}(\Lambda^0)$ emits an edge of each color, because the same is true in $\Lambda$; if $v$ emits an edge $ g \in \E^{e_1}$ then $s_D(g^1) = \iota^{-1}(s(g)) = v$, and all other edges emitted by $\iota(v)$ are in $\iota(\Lambda_D^1)$ and hence also occur in  $\Lambda_D$.

Furthermore, the number of edges with range $v$ in $\Lambda_D$ is the same as the number of edges with range $\iota(v) \in \Lambda$, if $v \not= v_g$.  In this setting, an edge in $v \Lambda_D^1 \backslash \iota^{-1}(v\Lambda^1_D)$ is necessarily of the form $g^2$ for some $g \in v\E^{e_1}$, so 
\[ | r_D^{-1}(v)| = |v\E^{e_1}| + |\iota^{-1}(v\Lambda^1_D)| = |r^{-1}(\iota(v))|.\]
If $v = v_g$, then  $r_D^{-1}(v)$ is still finite as long as $\Lambda$ is row-finite:
\begin{align*} |r_D^{-1}(v_g) | &= \left| \{ g^1\} \cup \bigcup_{i=2}^k v_g \E^{e_i}_D \right| \leq 1 + \left| \{ \alpha \in \Lambda: \alpha = [g b] \text{ for some } b \in \Lambda^1 \} \right| \\
&= 1 + r^{-1}(s(g)) < \infty.  \end{align*}
We conclude that $\Lambda_D$ is a row-finite, source-free $k$-graph whenever $\Lambda$ is.
\end{proof}

\begin{theorem}
\label{thm:DelayME}
Let $(\Lambda, d)$ be a row-finite, source-free \textit{k}-graph and let $(\Lambda_D,d_D)$ be the graph of $\Lambda$ delayed at an edge $f$. Then $C^*(\Lambda_D)$ is Morita equivalent to $C^*(\Lambda)$.
\end{theorem}

\begin{proof}
Let $\{ t_\lambda : \lambda \in \Lambda_D^1 \cup \Lambda_D^0 \}$ be the canonical Cuntz--Krieger $\Lambda_{D}$-family generating $C^*(\Lambda_{D})$. Define 
\begin{equation*}
\begin{split}
&S_{v} = t_{v}, \text{ for } v \in \Lambda^{0}, \\
&S_{h} = \left\{ \begin{array}{ll}
t_{\iota^{-1}(h)} & \text{if } h \notin \E^{e_{1}}\\
t_{h^{2}}t_{h^1} & \text{if } h \in \E^{e_{1}}
\end{array} \right. \text{, for } h \in \Lambda^{1}.
\end{split}
\end{equation*}
\\

We claim that $\{S_\lambda : \lambda \in \Lambda^0 \cup \Lambda^1 \}$ is a Cuntz--Krieger $\Lambda$-family in $C^*(\Lambda_D)$. Note that since $\{t_v: v \in \Lambda^0_D \}$ are mutually orthogonal projections, so are $\{S_{v}: v \in \Lambda^0 \}$. Therefore $\{S_\lambda : \lambda \in \Lambda^0 \cup \Lambda^1 \}$ satisfies (CK1). Now take arbitrary $a,b,g,h \in \Lambda^1$ such that $ah \sim gb$. Assuming $a,b,g,h \notin \E^{e_1}$,  Case 1 of Definition \ref{def:delay} implies that
\begin{equation*}
S_a S_h = t_{\iota^{-1}(a)} t_{\iota^{-1}(h)} = t_{\iota^{-1}(g)} t_{\iota^{-1}(b)} = S_g S_b.
\end{equation*}
Conversely, suppose either $a,b \in \E^{e_1}$, or $g,h \in \E^{e_1}$. Without loss of generality assume $g,h \in \E^{e_1}$. Then  $\alpha:= [ah] \in \E_{D}^{e_{i}}$ satisfies $e_{\alpha} h^1 \sim_D g^1 b$ and $a h^2 \sim_D g^2 e_\alpha$. Hence
\begin{equation*}
\begin{split}
S_a S_h
= t_a t_{h^2} t_{h^1} 
= t_{g^2} t_{e_{\alpha}} t_{h^1} 
= t_{g^2} t_{g^1} t_b
= S_gS_b.
\end{split}
\end{equation*}
Therefore $\{S_\lambda : \lambda \in \Lambda^0 \cup \Lambda^1 \}$ satisfies (CK2). 

For (CK3), let $h \in \Lambda^1 \setminus \E^{e_{1}}$; observe that $S_{h} = t_{h}$, and hence $
S_h^* S_h = t_h^* t_h = t_{s_D(h)} = S_{s(h)}.$
If $h \in \E^{e_{1}}$, we similarly have $
S_h^* S_h = t_{h^1}^* t_{h^2}^* t_{h^2} t_{h^1} = 
 t_{s_D(h)} = S_{s(h)}.$
Therefore $\{S_\lambda : \lambda \in \Lambda^0 \cup \Lambda^1 \}$ satisfies (CK3). 

Finally, take  an arbitrary $v \in \Lambda^0$. Observe that if $h \in \E^{e_1}$, then $h^1$ is the only edge in $\Lambda_D$ such that $d_D(h_1) = e_1$ and $r_D(h_1) = v_h$. Thus $t_{v_h} = t_{h^1} t_{h^1}^*$ and consequently
\begin{equation*}
\begin{split}
\sum_{h \in r^{-1}(v) \cap \Lambda^{e_1}} S_h S_h^* 
&= \sum_{\substack{h \in r^{-1}(v) \cap \Lambda^{e_1} \\ h \notin \E^{e_1}}} t_{\iota^{-1}(h)} t_{\iota^{-1}(h)}^* + \sum_{\substack{h \in r^{-1}(v) \cap \Lambda^{e_1} \\ h \in \E^{e_1}}} t_{h^2} t_{h^1} t_{h^1}^* t_{h^2}^* \\
&= \sum_{\substack{h \in r^{-1}(v) \cap \Lambda^{e_{1}} \\ h \notin \E^{e_1}}} t_{\iota^{-1}(h)} t_{\iota^{-1}(h)}^* + \sum_{\substack{h \in r^{-1}(v) \cap \Lambda^{e_{1}} \\ h \in \E^{e_1}}} t_{h^2} t_{h^2}^* \\
&= \sum_{e \in r^{-1}_D(v) \cap \Lambda^{e_{1}}_D}t_{e}t_{e}^*
= t_{v} 
= S_{v}.
\end{split}
\end{equation*}
Now, if $2 \leq i \leq k$, the fact that $v \in \Lambda^0$ means that  $\iota^{-1}(v) \Lambda^{e_i}_D = \iota^{-1}(v \Lambda^{e_i})$: there are no edges of the form $h^j$ or $e_\alpha$ with range $v$ and degree $e_i$.  Consequently,
\begin{equation*}
\sum_{h \in \Lambda^{e_{i}}} S_{h}S_{h}^*
= \sum_{h \in v\Lambda^{e_{i}}} t_{h}t_{h}^* 
= t_{v} 
= S_{v}. 
\end{equation*}
Therefore $\{S_\lambda : \lambda \in \Lambda^0 \cup \Lambda^1 \}$ satisfies (CK4), and hence is a Cuntz--Krieger $\Lambda$-family in $C^*(\Lambda_{D})$. By the universal property of $C^*(\Lambda)$, there exists a homomorphism $\pi:C^*(\Lambda) \rightarrow C^*(\Lambda_D)$.

To see that that $\pi$ is injective, we use Theorem \ref{thm:GIUT} and Lemma \ref{lemma:Action}. Define $\beta: \mathbb{T}^{k} \to \text{Aut}(C^*(\Lambda_D))$ by setting, for all $z \in {\mathbb T}^k$,
\begin{equation*}
\begin{split}
  &\beta_z(t_h) = z^{d(h)} t_h \text{, for } h \neq g^1, \\
  &\beta_z(t_h) = t_h \text{, for } h = g^1, \\
  &\beta_z(t_v) = t_v \text{, for } v \in \Lambda^0,
\end{split}
\end{equation*}
and extending $\beta$ to be linear and multiplicative.
By taking $E =\{ g^1: g \in \E^{e_1} \}$ and applying Lemma \ref{lemma:Action}, we see that $\beta$ is an action of $\mathbb{T}^k$ on $C^*(\Lambda_D)$. Let $\alpha$ be the canonical gauge action on $C^*(\Lambda) = C^*(\{ s_\lambda: \lambda \in \Lambda^0 \cup \Lambda^1\})$ and note that for $e \notin \E^{e_1}$ we have
\begin{equation*}
\pi[\alpha_z(s_{e})] = \pi[z^{d(e)}s_{e}] = z^{d(e)}S_{e} = z^{d(e)}t_{e} = \beta_z(t_{e}) = \beta_z(T_{e}) = \beta_z(\pi[s_{e}])
\end{equation*}
and for $e \in \E^{e_1}$,
\begin{equation*}
\pi[\alpha_z(s_{e})] = \pi[z^{d(e)}s_{e}] = z^{d(e)}S_{e} = z^{d(e)}t_{e^2}t_{e^1} = \beta_z(t_{e^2}t_{e^1}) = \beta_z(T_{e}) = \beta_z(\pi[s_{e}]). 
\end{equation*}
It is straightforward to check that $\alpha$ and $\beta$ commute on the vertex projections. Therefore, $\beta$ commutes with the canonical gauge action, so Theorem \ref{thm:GIUT} implies that $\pi$ is injective. 

To see that $\text{Im}(\pi) \cong C^*(\Lambda)$ is Morita equivalent to $C^*(\Lambda_D)$, we invoke Theorem \ref{thm:allen}.   Set $X = \iota^{-1}(\Lambda^0) \subseteq \Lambda^0_D$; we will show that the saturation $\Sigma(X)$ of $X$ is  $\Lambda^0_D$.  If  $g \in \E^{e_1}$, then $r(g) \in \iota(\Lambda_D^0)$ and $g^2 \in \iota^{-1}(r(g)) \Lambda_D$.  Therefore if $H$ is hereditary and contains $X$, we must have $s_D(g^2) = v_g \in H$ for all $g \in \E^{e_1}$.  Consequently, $H = \Lambda_D^0$.  Since $\Lambda^0_D$ is evidently saturated, we have $\Sigma(X) = \Lambda^0_D$ as claimed.  Theorem \ref{thm:allen} therefore implies that 
\[ P_X C^*(\Lambda_D) P_X \cong_{ME} C^*(\Lambda_D).\]

We will now complete the proof that $C^*(\Lambda) \cong_{ME} C^*(\Lambda_D)$ by showing that 
\[ P_X C^*(\Lambda_D) P_X = \text{Im}(\pi)\cong C^*(\Lambda).\]  The generators $\{ S_\lambda: \lambda \in \Lambda^0 \cup  \Lambda^1\}$ of $\text{Im}(\pi)$ all satisfy $P_X S_h P_X = S_h$, so $\text{Im}(\pi) \subseteq P_X C^*(\Lambda_D) P_X$. For the other inclusion, note that 
\[ P_X C^*(\Lambda_D) P_X = \overline{\text{span}}\, \{ t_\lambda t_\mu^*: \lambda, \mu \in G_D^*,\ s_D(\lambda) = s_D(\mu) ,\  r(\lambda), r(\mu) \in X\}.\]
Given $\lambda \in G_D^*$ with $r_D(\lambda) \in X$, create $\lambda' \in [\lambda]_D$ by first replacing any path of the form $g^2 e_{[ah]}$ in $\lambda$ with its equivalent $a h^2$, and then replacing paths of the form $e_{[gb]} h^1$ with their equivalent  $g^1 b$. Note  that since $r(\lambda) = r_D(\lambda) \in X$ we cannot have $\lambda_{|\lambda|} \in \E^{e_i}_D$ for any $i\not= 1.$  Because of this, $\lambda'$ will contain no edges in $\bigcup_{i=2}^k  \E^{e_i}_D$.

Thus, in $\lambda'$, any occurrence of an edge of the form $g^2$ will be preceded by $g^1$ unless $s_D(\lambda') \not \in X$ (in which case, if $s_D(\lambda') = v_g, \ \lambda'_1 = g^2$). 
Consequently, if $s_D(\lambda') \in X$, then $ t_{\lambda'_{|\lambda|}} \cdots t_{\lambda'_2} t_{\lambda'_1} = t_{\lambda'} = t_{\lambda}$ is a product of operators of the form $S_h, S_k^*$ for $h, k \in \Lambda^1$.

Similarly, given $\mu \in G^*_D$ with $r_D(\mu) \in X$ and $s_D(\mu) = s_D(\lambda)$, create $\mu' \in [\mu]_D$ by the procedure above, so that $\mu'$ contains no edges in $\bigcup_{i=2}^k \E^{e_i}_D$ and any edge in $\mu'$ of the form $g^2$ (with the possible exception of $\mu'_1$) is preceded by $g^1$.  It follows that for any  $\mu, \lambda$ with  $r_D(\mu), r_D(\lambda) \in X$ and $s_D(\mu) = s_D(\lambda) \in X$ we have $t_\lambda t_\mu^* \in \text{Im}(\pi)$.

If $t_\lambda t_\mu^* \in P_X C^*(\Lambda_D) P_X$ and $s_D(\lambda)\not\in X$, write $v_g = s_D(\lambda)$. As observed earlier, in this case we must have $\lambda'_1 = \mu'_1 = g^2.$  Moreover, since $v_g$ receives precisely one edge of degree $e_1$ (namely $g^1$) we have $p_{v_g} = t_{g^1} t_{g^1}^*$.  It follows that 
\[ t_\lambda t_\mu^* = t_{\lambda'} t_{\mu'}^* = t_{\lambda' g^1} t_{\mu' g^1}^*.\]
Observe that neither $\lambda' g^1$ nor $\mu' g^1$ contains any edge in $\bigcup_{i=2}^k \E^{e_i}_D$, and every occurrence of an edge of the form $h^2$ in either path is preceded by $h^1$.  Thus, in this case as well we can write $t_\lambda t_\mu^*$ as a product of operators of the form  $S_h, S_k^*$.  

Since $\text{Im}(\pi)$ is norm-closed, it follows that every element in  $\overline{\text{span}} \{ t_\mu t_\lambda^*: \lambda, \mu \in \Lambda_D, r(\lambda) = r(\mu) \in X \} = P_X  C^*(\Lambda_D)P_X$ lies in $\text{Im}(\pi)$.
Thus, 
\[C^*(\Lambda) \cong \text{Im}(\pi) = P_X C^*(\Lambda_D) P_X \cong_{ME} C^*(\Lambda_D),\] as claimed.
\end{proof}

\section{Sink Deletion} \label{sec:Sink}

In this section, we analyze the effect on $C^*(\Lambda)$ of deleting a {\em sink}  -- a vertex which emits no edges of a certain color -- from $\Lambda$.  This should be viewed as the analogue of Move \texttt{(S)}, removing a regular source, for directed graphs, as the conventions used to define a Cuntz--Krieger family in \cite{errs, sorensen-first} differ from the standard conventions for higher-rank graph $C^*$-algebras.  We show in Theorem \ref{thm:SinkKG} 
that the result of deleting a sink from a $k$-graph is still a $k$-graph, and Theorem \ref{thm:SinkME} shows that the resulting $C^*$-algebra is Morita equivalent to the original  $k$-graph $C^*$-algebra.

\begin{definition}\label{def:Sink} Let $\Lambda$ be a \textit{k}-graph.
We say $v \in \Lambda^0$ is an $e_i$ \textit{sink} if $s^{-1}(v) \cap \Lambda^{e_i} = \emptyset$ for some $1 \leq i \leq k$. We say $v$ is a \textit{sink} if it is an $e_i$ sink for some $1 \leq i \leq k$.
\end{definition}

\begin{definition}\label{def:Sink}
Let $(\Lambda, d)$ be a \textit{k}-graph. Let $G = (\Lambda^0, \Lambda^1, r, s)$, and $G^*$ be its 1-skeleton and category of paths respectively. Let $v \in \Lambda^0$ be a sink. We write $w \leq v$ if there exists a $\lambda \in G^*$ such that $s(\lambda) = v$ and $r(\lambda) = w$. Define the directed colored graph $G_S = (\Lambda^0_S, \Lambda^1_S, r_S, s_S)$ by 
\begin{equation*}
\begin{split}
& \Lambda^0_S := \{ w: w \not\leq v\}, \qquad 
\Lambda^1_S : =  \Lambda^1 \setminus \{ f \in \Lambda^1:  r(f) \leq v \};
\end{split}
\end{equation*}
we set $r_S = r, s_S = s, d_S = d$.
Let $\iota: G^*_S \to G^*$ be the inclusion map, and define an equivalence relation on $G_S^*$ by
$\mu \sim \lambda$ when $[\iota(\mu)] = [\iota(\lambda)] \in \Lambda$. Define $\Lambda_S = G^* / \sim$ and call $\Lambda_S$ the \textit{k}-\textit{graph of} $\Lambda$ \textit{with the sink }$v$ \textit{deleted.}
\end{definition}

\begin{example}
The graphs $G$ and $G_S$ after deleting the blue sink $v$, where blue is the dashed color.
\begin{center}
\begin{tikzcd}[column sep=1cm,row sep=1cm]
G & v \ar[loop above, red] \ar[d, red] & \ar[l, dashed, blue] \bullet \ar[loop above, red] \ar[d, red] \ar[loop right, dashed, blue]\\
\ar[loop left, dashed, blue] \bullet \ar[r, dashed, blue] & \bullet & \ar[l, dashed, blue] \bullet \ar[loop right, dashed, blue]\\
\ar[loop left, dashed, blue] \bullet \ar[u, red] \ar[r, dashed, blue] \ar[loop below, red] & \bullet \ar[u, red] \ar[loop below, red]
\end{tikzcd}
\hspace{10mm}
\begin{tikzcd}[column sep=1cm,row sep=1cm]
G_S &  & \bullet \ar[loop above, red] \ar[d, red] \ar[loop right, dashed, blue]\\
\ar[loop left, dashed, blue] \bullet & & \bullet \ar[loop right, dashed, blue]\\
\ar[loop left, dashed, blue] \bullet \ar[u, red] \ar[r, dashed, blue] \ar[loop below, red] & \bullet \ar[loop below, red]
\end{tikzcd}
\end{center}
\end{example}
\begin{example} \label{ex:sinkintroduction} The following example highlights the fact that performing a sink-deletion may introduce new sinks.  When $\Lambda$ has a finite vertex set, performing successive sink deletions will eventually produce a sink-free \textit{k}-graph.
 \begin{figure}[H]
\label{fig:sinkintroduction}
\begin{tikzcd}[column sep=1cm,row sep=.5cm]
G \\
\ar[out=135, in=-135, distance=4em,red] \ar[out=150, in=-150, distance=2em, dashed, blue] \bullet \ar[r, bend left, red] \ar[r, bend right, dashed, blue]
& w \ar[r, bend left, red] \ar[r, bend right, dashed, blue]
& v
\end{tikzcd}
\hspace{10mm}
\begin{tikzcd}[column sep=1cm,row sep=.5cm]
G_S \\
\ar[out=135, in=-135, distance=4em,red] \ar[out=150, in=-150, distance=2em, dashed, blue] \bullet \ar[r, bend left, red] \ar[r, bend right, dashed, blue]
& w 
\end{tikzcd}
\caption{Sink deletion at $v$ creating a new sink at $w$.}
\end{figure}
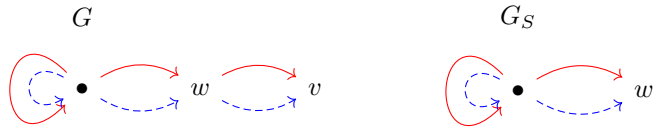
\end{example}
\begin{lemma}\label{lem:Sink}
If $(\Lambda, d)$ is a \textit{k}-graph with $v \in \Lambda^0$ an $e_i$ sink, then  $ \{ x  \in \Lambda^0: x \leq v \}$ consists of $e_i$ sinks.
\end{lemma}

\begin{proof}
If  $w \leq v $  and $w$ is not an $e_i$ sink, then there exists a $y \in \Lambda^0$ and $f \in \Lambda^{e_i}$ such that $s(f) = w$ and $r(f) = y$. Thus there exists a path $f \lambda \in s^{-1}(v) \cap r^{-1}(y).$ Furthermore, since $\Lambda$ is a \textit{k}-graph there exists a path $\mu g \sim f \lambda$ with $g \in \Lambda^{e_i}$. That is, $v$ is not an $e_i$ sink.
\end{proof}

\begin{theorem}\label{thm:SinkKG}
If $(\Lambda, d)$ is a source-free \textit{k}-graph with $v \in \Lambda^0$ a sink then $(\Lambda_S, d_S)$, the graph of $\Lambda$ with the sink $v \in \Lambda^0$ deleted, is a source-free \textit{k}-graph.
\end{theorem}

\begin{proof}
Take $\lambda \in G_S^*$. Since $r(\iota(\lambda)) \not\leq v$, if  $\mu = \mu_1 \cdots \mu_n \sim  \iota(\lambda)$, then $r(\mu) \not\leq v$.  In fact,  we have $s(\mu) = s(\iota(\lambda)) \not \leq v$, and  $r(\mu_i) \notin V$ for all $1 \leq i \leq n$.  To see this, simply recall that if $s(\eta) \leq v$ then $r(\eta) \leq v$ as well.   Consequently,
 $\mu \in \iota(G_S^*)$ and $\iota^{-1}(\mu) \sim_S \lambda$.  Thus $[\lambda]_S = [\iota(\lambda)]$, which satisfies (KG0) and (KG4) because $\Lambda$ is a $k$-graph. Since $\lambda \in G_S^*$ was arbitrary, it follows that  $\Lambda_S$ is a $k$-graph.
 
 To see that $\Lambda_S$ is source-free, note that 
whenever an edge $e \in G^1$ was deleted in the process of forming $\Lambda_S$, so too was the vertex $r(e) \in G^0$. Therefore, no sources were created in the formation of $\Lambda_S$, so the $k$-graph $\Lambda_S$ is source-free.
\end{proof}

\begin{theorem}\label{thm:SinkME}
If $(\Lambda, d)$ is a source free row finite \textit{k}-graph with $v \in \Lambda^0$ a sink and $(\Lambda_S,d_S)$ the \textit{k}-graph of $\Lambda$ with $v$ deleted, then $C^*(\Lambda)$ is Morita equivalent to $C^*(\Lambda_S)$.
\end{theorem}

\begin{proof}
Let $\{ s_\lambda : \lambda \in \Lambda^0 \cup \Lambda^1 \}$ be the canonical Cuntz-Krieger $\Lambda$-family in $C^*(\Lambda)$.  Then for every $\lambda \in \Lambda^0_S \cup \Lambda^1_S$ define 
\begin{equation*}
T_\lambda = s_{\iota(\lambda)}. 
\end{equation*}
We first prove that $\{T_\lambda : \lambda \in \Lambda_S^0 \cup \Lambda^1_S \}$ is a Cuntz-Krieger $\Lambda_S$-family in $C^*(\Lambda)$. Note that $\{ s_x : x \in \Lambda^0 \}$ are non-zero and mutually orthogonal, and thus so are $\{ T_x : x \in \Lambda_S^0 \}$. Therefore $\{T_\lambda : \lambda \in \Lambda_S^0 \cup \Lambda^1_S \}$ satisfies (CK1). Since $\{s_\lambda : \lambda \in \Lambda^0 \cup \Lambda^1 \}$ is a Cuntz-Krieger $\Lambda$-family in $C^*(\Lambda)$, the fact that $fg \sim_S h j $ iff $\iota (fg) = \iota(f) \iota(g) \sim \iota(h) \iota (j) = \iota(hj)$ tells us that if  for $fg \sim_S hj $, then 
\begin{equation*}
    T_f T_g = s_{\iota(f)} s_{\iota(g)} = s_{\iota(h)} s_{\iota(k)} = T_h T_k,
\end{equation*}
and therefore $\{T_\lambda : \lambda \in \Lambda_S^0 \cup \Lambda^1_S \}$ satisfies (CK2). Also, for $f \in \Lambda^1$ we have
\begin{equation*}
T_f^*T_f 
= s_{\iota(f)}^* s_{\iota(f)} 
= s_{\iota(s(f))} 
= T_{s(f)},
\end{equation*}
and therefore $\{T_\lambda : \lambda \in \Lambda_S^0 \cup \Lambda^1_S \}$ satisfies (CK3). Finally note that for every $f \in \Lambda^{1}$, if $r(f) \not\leq v$ then $s(f) \not\leq v$. Thus for every $x \in \Lambda_S^0$, since $x$ was not deleted, $r^{-1}(\iota (x)) = \iota (r_S^{-1}(x))$. So, for every basis vector $e_i$ of $\N^k$ we have
\begin{equation*}
T_x = s_{\iota(x)}
= \sum\limits_{\substack{d{(\iota(\lambda))} = e_{i} \\ r(\iota(\lambda)) = x }} s_{\iota(\lambda)}s^*_{\iota(\lambda)}
= \sum\limits_{\substack{d{_S (\lambda)} = e_{i} \\ r_S(\lambda) = x }} T_{\lambda}T^*_{\lambda}.
\end{equation*}
Thus (CK4) is satisfied, so $\{T_\lambda : \lambda \in \Lambda^0_S \cup \Lambda^1_S \}$ is a Cuntz-Krieger $\Lambda_S$ family in $C^*(\Lambda)$. By the universal property of $C^*(\Lambda_S)$, then, there exists a $*$-homomorphism $\pi:C^*(\Lambda_S) \rightarrow C^*(\Lambda)$ such that $\pi(t_\lambda) = T_\lambda$ for any $\lambda \in \Lambda^0_S \cup \Lambda^1_S$. Observe that $\pi$ commutes with the canonical gauge actions on $C^*(\Lambda)$ and $C^*(\Lambda_S)$; moreover, $\pi(t_x) \not= 0$ for any $x \in \Lambda^0_S$. Consequently,  the gauge invariant uniqueness theorem (Theorem \ref{thm:GIUT}) tells us that $\pi$ is injective.

We now invoke Theorem \ref{thm:allen} to show that $\text{Im}(\pi) \cong_{ME} C^*(\Lambda)$. Consider $X = \iota(\Lambda_S^0) \subseteq \Lambda^0$, and 
set $p = \sum\limits_{x \in \Lambda_S^0} p_{\iota(x)}$. 
We claim that $\Sigma(X) = \Lambda^0$, so that Theorem \ref{thm:allen} implies that $p C^*(\Lambda) p \cong_{ME} C^*(\Lambda)$.  To see this, recall from Lemma \ref{lem:Sink} that every vertex in $\Lambda^0 \backslash X$ is an $e_i$ sink.  Moreover, the fact that $\Lambda$ is source-free implies that if $w \in \Lambda^0$ then $w \Lambda^{e_i}$ is nonempty.    Since $s(w\Lambda^{e_i}) \subseteq X$, it follows that every $w \in \Lambda^0$ lies in $\Sigma(X)$, as claimed.  

We now show that $p C^*(\Lambda) p \cong \text{Im}(\pi)$.  To that end, observe that
\begin{align*}
 p C^*(\Lambda) p&  = \overline{\text{span}} \{ s_\lambda s_\mu^*: r(\lambda), r(\mu) \in X = \iota(\Lambda_S^0)\} =\overline{\text{span}} \{ s_\lambda s_\mu^*: r(\lambda), r(\mu) \not\leq v\} .
 \end{align*}
Moreover, if  $r(\lambda) \not\leq  v$ then we must have $s(\lambda) \not \leq v$.
It follows that if $s_\lambda s_\mu^* \in p C^*(\Lambda) p$, then $s_\lambda, s_\mu \in\text{Im}(\pi)$.  Similarly, every generator $s_\lambda$ of $\text{Im}(\pi)$ lies in $p C^*(\Lambda) p$.  We conclude that, as desired, 
\[  C^*(\Lambda_S) \cong \text{Im}(\pi) = p C^*(\Lambda) p \cong_{ME} C^*(\Lambda). \qedhere\]
\end{proof}

\section{Reduction} \label{sec:Reduction}

In the geometric classification of unital graph $C^*$-algebras, the ``delay'' operation does not appear.  Instead, we find its quasi-inverse {\em reduction} in the  final list \cite{errs} of moves on graphs which encode all Morita equivalences between graph $C^*$-algebras.   Indeed, reduction -- rather than delay -- was a central ingredient in \cite{sorensen-first}, and it is more easily recognized as a special case of the general result of \cite{crisp-gow}.

For directed graphs, any delay can be undone by a reduction.
As we will see in the following pages, however, reduction for higher-rank graphs is not evidently an inverse to the ``delay'' move discussed in Section \ref{sec:delay}. For this reason we have elected to include a detailed treatment of both moves.

For row-finite directed graphs, reduction contracts an edge $e$ to its source vertex $v$, and can occur whenever $s^{-1}(v) = \{ e\}$ and all edges with range $v$ emanate from the same vertex $x \not= v$.  In the setting of higher-rank graphs, we can only reduce {\em complete edges} (see Notation \ref{not:CompleteEdge} below) which emanate from a vertex $v$ such that $r^{-1}(v)$ is also a complete edge.  Under these restrictions, however, reduction of a complete edge in $\Lambda$ results in a new $k$-graph $\Lambda_R$ such that $C^*(\Lambda) \cong_{ME} C^*(\Lambda_R)$. (See Theorems \ref{thm:ReductionKG} and \ref{thm:ReductionME} below.)

\begin{notation}\label{not:CompleteEdge}
Let $(\Lambda, d)$ be a $k$-graph.
We say a collection of edges, $E \subseteq \Lambda^1$, is a {\em complete edge} if it has the following three properties:
\begin{enumerate}
\item $E$ contains precisely one edge of each color;
\item  $s(e) = s(f)$ and $r(e) = r(f)$ for every $e,f \in E$; 
\item  if $e \in E$ and $a, b, f \in \Lambda^1$ satisfy $ea \sim fb$ or $ae \sim bf$, then $f \in E$.
\end{enumerate}
\end{notation}

\begin{example}
The third condition in Notation \ref{not:CompleteEdge} depends on the factorization rules.  For example, consider the edge-colored directed graph below.

\[
\begin{tikzpicture}[scale=1.5]
\node[inner sep=0.5pt, circle] (v1) at (0,0) {$v$};
\node[inner sep=0.5pt, circle] (v2) at (1.5,0) {$w$};
\draw[-latex, red, dashed] (v1) edge [out=30, in=150] (v2); 
\draw[-latex, blue] (v1) edge [out=-30, in=210] (v2); 
\node at (0.65, 0.12) {\color{black} $e_2$};
\node at (0.85, -0.12) {\color{black} $f_2$};
\draw[-latex, blue] (v1) edge [out=140, in=190, loop, min distance=15, looseness=2.5] (v1);
\draw[-latex, red, dashed] (v1) edge [out=120, in=210, loop, min distance=40, looseness=2.5] (v1);
\draw[-latex, red, dashed] (v2) edge [out=-40, in=10, loop, min distance=15, looseness=2.5] (v2);
\draw[-latex, blue] (v2) edge [out=-60, in=30, loop, min distance=40, looseness=2.5] (v2);
\node at (-0.6, 0.1) {$f_1$}; \node at (-1,0.3) {$e_1$}; \node at (2.45, 0.2) {$f_3$}; \node at (2.05, 0) {$e_3$};
\end{tikzpicture}
\]
If we define $f_2 e_1 \sim e_2 f_1$ and $f_3 e_2 \sim e_3 f_2$, then each set $\{ e_i, f_i\}$ is a complete edge, for $i= 1,2,3$.  However, if we instead define $f_2 e_1 \sim e_3 f_2$ and $f_3 e_2 \sim e_2 f_1$, then there are no complete edges.
\end{example}

\begin{definition}\label{def:Reduction}
Let $(\Lambda, d)$ be a $k$-graph and $G = (\Lambda^0, \Lambda^1, r, s)$ its 1-skeleton.  Fix $v \in \Lambda^0$ such that both $ \Lambda^1 v$ and $v \Lambda^1$ are complete, and such that $v \not =  r( \Lambda^1 v)=: w $. Define the directed colored graph $G_R = (\Lambda_R^0, \Lambda_R^1, r_R, s_R)$ by 
\begin{equation*}
\begin{split}
&\Lambda_R^0 = \Lambda^0 \setminus \{ v \}, \\
&\Lambda_R^1 = \Lambda^1 \setminus  \Lambda^1 v, \\
&s_R(e) = s(e), \\
&r_R(e) = \left\{ \begin{array}{ll} 
r(e) & \text{if } r(e) \neq v \\
w & \text{if } r(e) = v.
\end{array} \right.
\end{split}
\end{equation*}
As the vertices and edges of $G_R$ are subsets of the vertices and edges of $G$, we write $\iota: \Lambda_R^1 \cup \Lambda_R^0  \to \Lambda^0 \cup \Lambda^1$ for the inclusion map.

Let $G_R^*$ be the path category of $G_R$; we will define a parent function $par: G_R^* \rightarrow G^*$. To that end, fix an edge $f \in  \Lambda^1 v$ and define
\begin{equation*}
\begin{split}
& par(x) = \iota(x) \text{, for } x \in \Lambda_R^0, \\
& par(e) = \left\{ \begin{array}{ll}
\iota(e) & \text{if } r(\iota(e)) \neq v \\
f \, \iota(e) & \text{if } r(\iota(e)) = v
\end{array} \right.
\text{, for } e \in \Lambda_R^1 , \\
& par(\lambda) = par(\lambda_{|\lambda|}) \cdots par(\lambda_1) \text{, for } \lambda = \lambda_{|\lambda|} \cdots \lambda_2 \lambda_1 \in G_R^* .\\
\end{split}
\end{equation*}
Then define the degree map $d_{R}$ on $G_R^*$ such that $d_{R}(e) = d(\iota(e))$. 
Define an equivalence relation, $\sim_R$, on $G_R^*$ by $\mu \sim_R \lambda$ if $par(\mu) \sim  par(\lambda)$. Let $\Lambda_R = G_R^* / \sim_R$; we  call $\Lambda_R$ the {\em graph of $\Lambda$ reduced at} $v \in \Lambda^0$.  
\end{definition}

\begin{example}
 Figures \ref{ex:Reduction1} and \ref{ex:Reduction2} show the result of reduction at a vertex $v$ in  two different $k$-graphs.  In both cases, we only picture the underlying 1-skeleton, as we have no choice in the factorization.

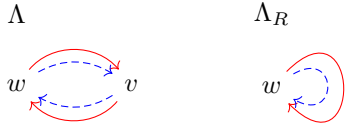
\begin{figure}[H]
\begin{tikzcd}[column sep=1cm,row sep=.5cm]
\Lambda \\
w \ar[r, bend left=50, red] \ar[r, bend left=30, dashed, blue]
& v \ar[l, bend left=50, red] \ar[l,bend left=30, dashed, blue]
\end{tikzcd}
\hspace{10mm}
\begin{tikzcd}[column sep=1cm,row sep=.5cm]
\Lambda_{R} \\
w \ar[out=45, in=-45, distance=4em, red] \ar[out=30, in=-30, distance=2em, dashed, blue]
\end{tikzcd}
    \caption{First example of reduction}
    \label{ex:Reduction1}
\end{figure}

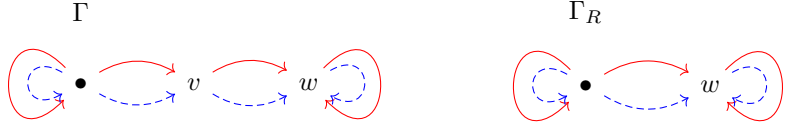
\begin{figure}[H]
\begin{tikzcd}[column sep=1cm,row sep=.5cm]
\Gamma \\
\ar[out=135, in=-135, distance=4em,red] \ar[out=150, in=-150, distance=2em, dashed, blue] \bullet \ar[r, bend left, red] \ar[r, bend right, dashed, blue]
& v \ar[r, bend left, red] \ar[r, bend right, dashed, blue]
& w \ar[out=45, in=-45, distance=4em, red] \ar[out=30, in=-30, distance=2em, dashed, blue]
\end{tikzcd}
\hspace{10mm}
\begin{tikzcd}[column sep=1cm,row sep=.5cm]
\Gamma_{R} \\
\ar[out=135, in=-135, distance=4em,red] \ar[out=150, in=-150, distance=2em, dashed, blue] \bullet \ar[r, bend left, red] \ar[r, bend right, dashed, blue]
& w \ar[out=45, in=-45, distance=4em, red] \ar[out=30, in=-30, distance=2em, dashed, blue]
\end{tikzcd}
    \caption{Second example of reduction.}
    \label{ex:Reduction2}
\end{figure}
\end{example}
\begin{theorem}\label{thm:ReductionKG}
If $(\Lambda, d)$ is a row-finite, source-free $k$-graph then $(\Lambda_R, d_R)$, the graph of $\Lambda$ reduced at $v \in \Lambda^0$, is a row-finite source-free $k$-graph.
\end{theorem}

\begin{proof}
To see that $\sim_R$ satisfies (KG0), suppose that $\lambda = \lambda_2 \lambda_1 \in G_R^*$ and that $\mu_1, \mu_2 \in G_R^*$ satisfy $par(\lambda_i)\sim par(\mu_i)$.  Then the definition of the parent function, and the fact that $\sim$ satisfies (KG0), implies that 
\[ par(\mu_2 \mu_1) = par(\mu_2) par(\mu_1) \sim par(\lambda_2) par(\lambda_1) = par(\lambda).\]
It follows that $\mu_2 \mu_1 \sim_R \lambda$, which establishes (KG0).

 Now, take an arbitrary $\lambda \in G_R^*$ and suppose $\iota(\lambda) = par(\lambda)$.  It follows that $\iota(\lambda)$ never passes through $v$; the fact that both $v \Lambda^1$ and $\Lambda^1 v$ are complete edges therefore implies that no path in $[\iota(\lambda)]$ passes through $v$.
 As $\Lambda$ is a $k$-graph, $ [par(\lambda)]$  satisfies (KG4). Since  $[\lambda]_R =\iota^{-1}( 
[par(\lambda)])$ and $\iota$ is both injective, and onto $\{ \mu \in G^*: v \text{ not on } \mu\} \supseteq [par(\lambda)]$, it follows that $[\lambda]_R$ also satisfies (KG4).  

 If $\lambda = \lambda_{|\lambda|} \cdots \lambda_2 \lambda_1$ and $\iota(\lambda)  \not= par(\lambda)$, then there exists at least one index $1 \leq i \leq |\lambda|$ such that $r_R(\lambda_i) = w $, where $w$ is the range of the complete edge which was deleted to form $\Lambda_R$. For ease of notation, in what follows we will assume that there is only one such index $i$, but the same argument will work if there are several. Let the color order of $\lambda$ be $(m_1, \ldots , m_{|\lambda|})$; then the color order of $par(\lambda)$ is $(m_1, \ldots , m_i, d(f), \ldots , m_{|\lambda|})$. Now $[par(\lambda)]$ satisfies (KG4), so in particular, for each permutation $(n_1, \ldots, n_{|\lambda|})$ of $(m_1, \ldots , m_{|\lambda|})$, there exists a unique path $\mu' \in [par(\lambda)]$ such that 
\[ d(\mu'_j) = \begin{cases}
n_j, & j \leq i \\
d(f), & j = i+1 \\
n_{j-1}, & j > i+1.
\end{cases}  .\]
Since $\mu'$ and $par(\lambda)$ both have an edge of degree $d(f)$ in the $(i+1)$st position, and $\sim$ satisfies (KG0) and (KG1),  we must have $\mu'_{i+1} = (par(\lambda))_{i+1} = f$.  Thus,  $\mu' \in \text{Im}(par)$.  Setting $\mu = par^{-1}(\mu')$, we have $\mu \sim_R \lambda$. The fact that our permutation $(n_1, \ldots, n_{|\lambda|})$ was arbitrary implies that $[\lambda]_R$ includes a  path of every color order; the fact that $par$ is injective implies that such a path is unique.  Thus, $[\lambda]_R$ satisfies (KG4), so
Theorem \ref{thm:KGalternative} tells us that $\Lambda_R$ is a $k$-graph.

To see that $\Lambda_R$ is row-finite, it suffices to observe that $|x \Lambda_R^{e_i}| = |par(x) \Lambda^{e_i} | < \infty$ for all $x \in \Lambda_R^0$ and for all $1 \leq i\leq k$. The fact that $\Lambda_R$ is source-free follows from the analogous fact that $0 \not= | \Lambda^{e_i} par(x)| = |\Lambda^{e_i}_R x|$ for all $x$ and $i$.
\end{proof}

\begin{theorem}\label{thm:ReductionME}
If $(\Lambda, d)$ is a row-finite source-free  $k$-graph, with $(\Lambda_R, d_R)$ the graph of $\Lambda$ reduced at $v \in \Lambda^0$, then $C^*(\Lambda)$ is Morita equivalent to $C^*(\Lambda_R)$.
\end{theorem}

\begin{proof}
Let $\{ s_\lambda : \lambda \in \Lambda^0 \cup \Lambda^1 \}$ be the canonical Cuntz--Krieger $\Lambda$-family generating $C^*(\Lambda)$. 
Define 
\begin{equation*}
\begin{split}
T_\lambda = s_{par(\lambda)} \text{, for } \lambda \in \Lambda^0_R \cup \Lambda^1_R.
\end{split}
\end{equation*}

We first prove that $\{T_\lambda : \lambda \in \Lambda_R^0 \cup \Lambda_R^1 \}$ is a Cuntz--Krieger $\Lambda_R$-family in $C^*(\Lambda)$. Note that $\{ s_x : x \in \Lambda^0 \}$ are non-zero and mutually orthogonal, and thus so are $\{ T_x : x \in \Lambda_R^0 \}$. Thus $\{ T_\lambda : \lambda \in \Lambda^0_R \cup \Lambda^1_R \}$ satisfies (CK1). Further if $ab, cd \in G_R^*$ such that $ab \sim_R cd$, then $[ par(a) par(b)] =[par(ab)] = [par(cd)] = [par(c) par(d)]$. By (KG0), we therefore have 
\begin{equation*}
T_a T_b = s_{par(a)} s_{par(b)} = 
s_{par(c)} s_{par(d)}  = T_c T_d,
\end{equation*}
and therefore $\{ T_\lambda : \lambda \in \Lambda^0 \cup \Lambda^1 \}$ satisfies (CK2). Now fix $e \in \Lambda_R^1$.  If $r(\iota(e)) \not= v$, we have $\iota(e) = par(e)$ and hence
\begin{equation*}
T_e^* T_e = s_{par(e)}^* s_{par(e)} = s_{s(par(e))} = s_{par(s(e))} = T_{s(e)}.
\end{equation*}
If $r(\iota(e)) = v$, then we similarly have
\begin{equation*}
\begin{split}
T_e^* T_e 
&= (s_f s_{e})^* (s_f s_{e}) = s_{e}^* s_f^* s_f s_{e} = s_{e}^* s_v s_{e} = s_{s(e)} 
= T_{s(e)}.
\end{split}
\end{equation*}
Therefore $\{ T_\lambda : \lambda \in \Lambda^0_R \cup \Lambda^1_R \}$ satisfies (CK3).

Finally, to see that $\{ T_\lambda: \lambda \in \Lambda^0_R \cup \Lambda^1_R \}$ satisfies (CK4), we begin by considering  (CK4) for $x \in \Lambda_R^0$ such that $x \neq w (= r(\Lambda^1 v))$.  Note that for any basis vector $e_i \in \mathbb{N}^{k}$,
\begin{equation*}
\sum\limits_{\substack{r_R(e) = x \\ d_R(e) = e_i}} T_e T_e^*
= \sum\limits_{\substack{r_R(e) = x \\ d_R(e) = e_i}} s_{e} s_{e}^* 
= \sum\limits_{\substack{r(e) =x \\ d(e) = e_i}} s_e s_e^* 
= s_{x} = s_{par(x)} 
= T_x.
\end{equation*}
Thus, (CK4) holds for such vertices $x$.

In order to complete the proof that (CK4) holds, we first need a better understanding of the equivalence relation $\sim$ for paths which pass through $v\in \Lambda^0$.
Since $v \Lambda^1$  and $\Lambda^1 v$ are complete edges by hypothesis,  each set contains precisely one edge of each color. Thus, if we write  $g_{i} $ for the edge in $ v\Lambda^1$ with  $d(g_{i}) = e_{i}$ and $h_{i} $ for the edge in $\Lambda^1 v$  such that $d(h_{i}) = e_{i}$, then $s_v = s_{g_i} s_{g_i}^*$ 
for any $i$.  Moreover, by our hypothesis that $v\Lambda^1$ and $\Lambda^1 v$ are both complete edges, we have, for any $1 \leq i, j \leq k$,
\begin{equation*}
\begin{split}
h_j g_i \sim h_i g_j 
&\implies s_{h_j} s_{g_i} = s_{h_i} s_{g_j} \implies s_{h_j} s_{g_i} s_{g_i}^* = s_{h_i} s_{g_j} s_{g_i}^* \\
&\implies s_{h_j} = s_{h_i} s_{g_j} s_{g_i}^*.
\end{split}
\end{equation*}

Thus since $f = h_j$ for a unique $j$,
\begin{equation*}
s_f s_f^* 
= (s_{h_i}s _{g_j} s_{g_i}^*) (s_{h_i} s_{g_j} s_{g_i}^*)^* 
= s_{h_i} s_{g_j} s_{g_i}^* s_{g_i} s_{g_j}^* s_{h_i}^*
= s_{h_i} s_{h_i}^*
\end{equation*}
for all $i$. 
It follows that 
\begin{equation*}
\begin{split}
\sum\limits_{\substack{r_R(e) = w \\ d_R(e) = e_{i}}} T_e T_e^*
&= T_{g_i} T_{g_i}^* + \sum\limits_{\substack{r_R(e) = w \\ e \neq g_i \\ d_R(e) = e_{i}}} T_e T_e^* = (s_f s_{g_i}) (s_f s_{g_i})^* + \sum\limits_{\substack{e \in w \Lambda^{e_i} \\ s(e) \neq v }} s_e s_e^* \\
&= s_{h_i} s_{h_i}^* + \sum\limits_{\substack{e \in w\Lambda^{e_i} \\ s(e) \neq v }} s_e s_e^* = \sum\limits_{\substack{e \in w\Lambda^{e_i}}} s_e s_e^* 
= s_{w} 
= T_w.
\end{split}
\end{equation*}
\\
Therefore $\{ T_\lambda : \lambda \in \Lambda^0_R \cup \Lambda^1_R \}$ satisfies (CK4), and thus is a Cuntz--Krieger $\Lambda_R$-family in $C^*(\Lambda)$. By the universal property of $C^*(\Lambda_R),$ there exists a homomorphism $\pi:C^*(\Lambda_R) \rightarrow C^*(\Lambda)$ such that, if $C^*(\Lambda_R) = C^*(\{ t_\lambda: \lambda \in \Lambda^0_R \cup \Lambda^1_R\})$, we have $\pi(t_\lambda) = T_\lambda$ for all $\lambda \in \Lambda_R^0 \cup \Lambda_R^1$. The computation above shows that our choice of edge $f$ does not affect the validity of the construction.

We now use the gauge-invariant uniqueness theorem and Lemma \ref{lemma:Action} to
prove that $\pi$ is injective.  If $G$ is the 1-skeleton of $\Lambda$, we define $R:G^* \to \mathbb{Z}^k$ by
\begin{equation*}
\begin{split}
  & R(e) = d(e) \text{, for } e \in \Lambda^1 \text{ s.t. } s(e) \neq v, \\
  & R(e) = d(e) - d(f) \text{, for } e \in \Lambda^1 \text{ s.t. } s(e) = v, \\
  & R(\lambda) = \sum_{i = 1}^{|\lambda|}R(\lambda_i) \text{, for } \lambda = \lambda_{|\lambda|} \cdots \lambda_2 \lambda_1 \in G^* \\
  & R(x) = 0 \text{, for } x \in \Lambda^0.
\end{split}
\end{equation*}
To show that $R$ induces a well defined function on $\Lambda$, suppose that $\mu \sim \nu $ and consider $R(\mu), R(\nu)$. If $s(\mu_i) = v$ then the fact that $\Lambda^1 v$ is a complete edge implies that $s(\nu_i) = v$.  Therefore $R(\mu) = d(\mu) - l \cdot d(f)$ and $R(\nu) = d(\nu) - l \cdot d(f)$, where  $l \in \mathbb{N}$ counts the number of edges in $\mu$ with source $v$. Since $d(\mu) = d(\nu)$ we  conclude $R(\mu) = R(\nu)$. Thus, the function $\beta: \mathbb{T}^{k} \to \text{Aut}(C^*(\Lambda))$ defined by $\beta_z(s_{\mu}s_{\nu}^*) = z^{R(\mu) - R(\nu)} s_{\mu}s_{\nu}^*$ is an action by Lemma \ref{lemma:Action}.

 Let $\alpha$ be the canonical gauge action on $C^*(\Lambda_R)$.
For any $e \in \Lambda^1_R$, $s(\iota(e)) \neq v$ so $R(\iota(e)) = d(\iota(e)) = d_R(e)$. Moreover, if 
$r(\iota(e)) \not= v$, then $\iota(e) = par(e)$ and $\pi(t_e) = s_{\iota(e)}$, and so for any $z \in {\mathbb T}^k$,  
\begin{equation*}
\pi(\alpha_z(t_e))
= \pi(z^{d_R(e)}t_e)
= z^{d(\iota(e))}T_e 
= z^{d(\iota(e))}s_{\iota(e)} 
= z^{R(\iota(e))}s_{\iota(e)}
= \beta_z(s_{par(e)}) 
= \beta_z(\pi(t_e)).
\end{equation*}
If $r(\iota(e)) = v,$ we have 
\begin{equation*}
\begin{split}
\pi(\alpha_z(t_e)) 
= \pi(z^{d_R(e)}t_e)
= z^{d_R(e)}T_e 
= z^{d(\iota(e))} s_f s_{\iota(e)}
& = z^{R(f)}z^{R(e)} s_f s_{\iota(e)} \\
= \beta_z(s_f s_{\iota(e)}) 
= \beta_z(T_e) 
& = \beta_z(\pi(t_e)).
\end{split}
\end{equation*}
It is straightforward to check that $\alpha_z$ and $\beta_z$ also commute on the vertex projections. Therefore, $\pi$ intertwines $\beta$ with the canonical gauge action on $C^*(\Lambda_R)$ and thus, by the gauge invariant uniqueness theorem, $\pi$ is injective.

We now use Theorem \ref{thm:allen} to show that $\text{Im}\, (\pi) \cong C^*(\Lambda_R)$ is Morita equivalent to $C^*(\Lambda)$. 
Define $X:= \Lambda^0 \backslash \{ v\} 
$ and set $P_X = \sum_{x \in X} s_x \in M(C^*(\Lambda))$. Our first goal is to show that $P_X C^*(\Lambda) P_X = \text{Im}(\pi)$.

To see that $\text{Im}(\pi) \subseteq P_X C^*(\Lambda) P_X$, recall that, 
 if $\lambda \in \Lambda_R$, then the vertices $par(s_R(\lambda)) = s(par(\lambda)), par(r_R(\lambda)) = r(par(\lambda))$ both lie in $ X$.  It follows  that  $T_\lambda \in P_X C^*(\Lambda) P_X$ for all $\lambda \in \Lambda_R:$
\begin{equation*}
T_\lambda = s_{par(\lambda)} = s_{r(par(\lambda))} s_{par(\lambda)} s_{s(par(\lambda))} = P_X s_{par(\lambda)} P_X.
\end{equation*}
Thus, $\pi(C^*(\Lambda_R)) \subseteq P_X C^*(\Lambda) P_X.$

For the other inclusion, note that 
\[P_X C^*(\Lambda) P_X = \overline{\text{span}} \{ s_\lambda s_\mu^*  : \mu, \lambda \in G^*,\, s(\mu) = s(\lambda), \, r(\lambda), r(\mu)\in X \}.\]
  We will show that each such generator $s_\lambda s_\mu^*$ is in $\text{Im}(\pi)$.  We begin with the following special case.
  
  {\bf Claim 1:} $s_{h_i} s_{h_j}^* \in \text{Im}(\pi)$ for all edges $h_i , h_j \in \Lambda^1 v$.
  
  To see this, recall that $s_{h_i} = s_{h_j} s_{g_i} s_{g_j}^*$, and so 
  \begin{equation}
  \label{eq:hi-hj-*} 
  s_{h_i} s_{h_j}^* = s_{h_j} s_{g_i} (s_{h_j} s_{g_j})^*.
  \end{equation}
  Now, write $e_\ell = d(f)$. If $\ell = j$ then $s_{h_j} s_{g_i} = T_{\iota^{-1}(g_i)} \in \text{Im}(\pi)$ and $s_{h_j} s_{g_j} = T_{\iota^{-1}(g_j)}$, so $s_{h_i} s_{h_j}^* \in \text{Im}(\pi)$. Thus, we suppose $\ell \not= j$.   
  
  By appealing to the results of Section \ref{sec:Sink}, without loss of generality we may assume that $w = r(h_j)$ is not an $e_\ell $ sink.  Thus, there is an edge $e \in \Lambda^{e_\ell} w$.  Since $\Lambda^1 v$ is a complete edge, we must have $e h_j \sim h f$ for some edge $h$.
It follows that 
\[ s_{h_j} s_{g_i}  = s_w s_{h_j} s_{g_i} = s_e^* s_e s_{h_j} s_{g_i} = s_e^* s_h s_f s_{g_i} .\]
If $r(e) = r(h) \in X$, then $e = par(\iota^{-1}(e))$ and $h f g_i = par(\iota^{-1}(h) \iota^{-1}(g_i))$.  Consequently, in this case we have
\[ s_{h_j} s_{g_i} = 
 s_e^* T_{\iota^{-1}(h)} T_{\iota^{-1}(g_i)} = T_{\iota^{-1}(e)}^*T_{\iota^{-1}(h)} T_{\iota^{-1}(g_i)}   \in \text{Im}(\pi).\]
 If $r(e) = r(h) = v$, then writing $s_v = s_f^* s_f$ we have 
 \[ s_e^* s_{hf g_i} = (s_f s_e)^* s_{f h} s_{f g_i} = T_{\iota^{-1}(e)}^* T_{\iota^{-1}(h)} T_{\iota^{-1}(g_i)} \in \text{Im}(\pi).\]
 
Similarly, we compute  that $s_{h_j} s_{g_j} = T_{\iota^{-1}(e)}^* T_{\iota^{-1}(h)} T_{\iota^{-1}(g_j)} \in \text{Im}(\pi).$  Equation \eqref{eq:hi-hj-*} then implies that $s_{h_i} s_{h_j}^* \in \text{Im}(\pi)$ for all $1 \leq i, j \leq k$, so Claim 1 holds.

Next, we establish our second claim via a case-by-case analysis.

{\bf Claim 2:} If $\eta \in G^*$ and $s(\eta) \not= v$, then $s_\eta \in \text{Im}(\pi)$.  

To see this, note first that 
if  $v$ does not lie on $\eta$, then $\eta = par(\iota^{-1}(\eta))$, so $T_{\iota^{-1}(\eta)} = s_\eta \in P_X C^*(\Lambda) P_X$.  

 If $v$ lies on $\eta$, and $\eta \sim \nu$, then the fact that $\Lambda^1 v, v \Lambda^1$ are complete edges means that $\nu$ will also pass through $v$.  If $\nu \sim \eta$ is such that every edge of $\nu$  with  source $v$ is $f$, then $\nu \in \text{Im}(par)$ and therefore $s_\nu = s_\eta \in \text{Im}(\pi)$.

For the last case, Suppose that $v$ lies on $\eta$ but that for every path $\nu$ with $\nu \sim \eta$, there is an edge $\nu_i$ in $\nu$ with $s(\nu_i) = v$ and $\nu_i \not= f$.  Without loss of generality, suppose that $i$ is the smallest such.  By hypothesis,  $s(\eta) = s(\nu) \not= v$, so $i \not= 1.$

As established in the proof of Claim 1, $ s_{\nu_i} = s_e^* s_h s_f$
for edges $e$ of degree $d(f)$ and $h$ of degree $d(\nu_i)$, and we have 
\[ s_{\nu_i} s_{\nu_{i-1}} = s_e^* s_{h f \nu_{i-1}} \in \text{Im}(\pi).\]
By construction, $\nu_{i-2} \cdots \nu_1$ does not pass through $v$, and so $s_{\nu_{i-2} \cdots \nu_1} \in \text{Im}(\pi)$.  An inductive application of this argument now shows that $s_\nu = s_\eta$ lies in $\text{Im}(\pi)$ whenever $s(\eta) \not= v$.  This completes the proof of Claim 2.

Finally, consider an arbitrary generator $s_\lambda s_\mu^*$ of $P_X C^*(\Lambda) P_X$, with $\lambda = \lambda_{|\lambda|} \cdots \lambda_1, \mu = \mu_{|\mu|} \cdots \mu_1 \in G^*$.  If $s(\lambda) = s(\mu) \not= v$, then applying Claim 2 to $\lambda$ and $\mu$, we see that $s_\lambda s_\mu^* \in \text{Im}(\pi)$.
If $v = s(\lambda) = s(\mu),$ then by Claim 1, $s_{\lambda_1} s_{\mu_1}^* \in \text{Im}(\pi)$.  Since $r(\lambda_1) = r(\mu_1) = w \not= v$, it follows from Claim 2 that if $\eta := \lambda_{|\lambda|} \cdots \lambda_2$ and $\zeta := \mu_{|\mu|} \cdots \mu_2$, then  $s_\eta , s_\zeta \in \text{Im}(\pi)$.  Consequently, $s_\lambda s_\mu^* = s_\eta s_{\lambda_1} s_{\mu_1}^* s_\zeta^* \in \text{Im}(\pi)$ for any generator $s_\lambda s_\mu^*$ of $P_X C^*(\Lambda) P_X$.

Having established that $P_X C^*(\Lambda) P_X = \text{Im}(\pi) \cong C^*(\Lambda_R)$, we now compute the saturation $\Sigma(X)$ in order to apply Theorem \ref{thm:allen}.  Note that $w \in X$ and there are edges $h_j$ in $w \Lambda^1 v$, so any hereditary set containing $X$ must also contain $v$.  Thus $\Sigma(X) = \Lambda^0$.  Theorem \ref{thm:allen} therefore tells us that $C^*(\Lambda) \sim_{ME} C^*(\Lambda_R)$, as claimed.
\end{proof}

\bibliographystyle{amsalpha}
\bibliography{eagbib}

\end{document}